\documentclass{amsart}
\usepackage{oldgerm}
\usepackage{latexsym}
\usepackage{amsmath}
\usepackage{amscd}
\usepackage{amssymb} 
\usepackage{amsmath}
\usepackage{amsthm}
\usepackage{latexsym}

\def\(#1;#2){\left\{\begin{array} {l} #1 \\ #2 \end{array}\right.}

\def\listtwo(#1;#2;#3;#4){\left\{\begin{array}{cc} #1 & #2 \\ #3 &#4
			       \end{array}\right.}
			       
\def\listthree(#1;#2;#3;#4;#5;#6){\left\{\begin{array}{ll} #1 & #2 \\ #3 &#4
			        \\ #5 & #6 \end{array}\right.}
			       
\def\smallmattwo(#1;#2;#3;#4){\left(\begin{smallmatrix} #1 & #2 \\ #3 &#4
			       \end{smallmatrix}\right)}

\def\ssmallmattwo(#1;#2;#3;#4){\left.\begin{smallmatrix} #1 & #2 \\ #3 &#4
			       \end{smallmatrix}\right.}
			       
\def\mattwo(#1;#2;#3;#4){\left(\begin{array}{cc} #1 & #2 \\ #3 &#4
			       \end{array}\right)}

\def\(#1;#2){\left\{\begin{array} {l} #1 \\ #2 \end{array}\right.}

\def\vecttwo(#1;#2){\left(\begin{array}{c} #1 \\ #2
\end{array}\right)} 

\def\matthree(#1;#2;#3;#4;#5;#6;#7;#8;#9){\left(\begin{array}{ccc} 
#1 & #2 & #3 \\
 #4 & #5 &#6 \\
 #7 & #8 &#9
\end{array}\right)}
			       
\def\smallmatthree(#1;#2;#3;#4;#5;#6;#7;#8;#9){\left(\begin{smallmatrix}
#1 & #2 & #3 \\
 #4 & #5 &#6 \\
 #7 & #8 &#9
\end{smallmatrix}\right)}
				       
\def\vecttwo(#1;#2){\left(\begin{array}{c} #1 \\ #2
\end{array}\right)} 

\def\vectthree(#1;#2;#3){\left(\begin{array}{c} #1 \\ #2 \\ #3
\end{array}\right)}

\def\H{{\bf H}}
\def\C{{\bf C}}

\theoremstyle{plain} 

\newtheorem{thm}{Theorem}[section] 
\newtheorem{lem}[thm]{Lemma}
\newtheorem{prop}[thm]{Proposition}

\newtheorem{thms}{Theorem}[subsection] 
\newtheorem{lems}[thms]{Lemma}
\newtheorem{props}[thms]{Proposition}

{Corollary}[section]
\newtheorem*{xcor}{Corollary} 

\newtheorem{conj}{Conjecture} 

\newtheorem{prob}{Problem} 

\theoremstyle{definition}

\newtheorem*{xrem}{Remark}

\theoremstyle{remark} 

\newtheorem*{acknow}{Acknowledgments} 

\begin{document}
 
\title[Ikeda's conjecture]{ Ikeda's conjecture on the period  of the Duke-Imamo{\=g}lu-Ikeda lift }

\author{Hidenori KATSURADA and Hisa-aki KAWAMURA }

\date{}

\maketitle
\begin{abstract}
Let $k$ and $n$ be positive even integers. For a cuspidal Hecke eigenform $h$ in the Kohnen plus subspace of weight $k-n/2+1/2$ for $\varGamma_0(4),$ let $I_n(h)$ be the  Duke-Imamo{\=g}lu-Ikeda lift of $h$ in the space of cusp forms of weight $k$ for $Sp_n({\bf Z}),$ and $f$ the primitive form of weight $2k-n$ for $SL_2({\bf Z})$ corresponding to $h$ under the Shimura correspondence.
We then express the ratio $\displaystyle {\langle I_n(h), I_n(h) \rangle / \langle h, h \rangle} $ of the period of $I_n(h)$ to that of $h$ in terms of special values of certain $L$-functions of $f$. This proves the  conjecture proposed by Ikeda concerning the period of the Duke-Imamo{\=g}lu-Ikeda lift.
\end{abstract}
\footnote[0]{2010 {\it{Mathematics Subject Classification.}} Primary 11F67,
11F46, 11F66.}

\section{Introduction} 
One of the  fascinating problems in the theory of modular forms is to find the relation between the periods (or the Petersson products) of cuspidal  Hecke eigenforms which are related with each other through their $L$-functions.  In particular, there are several important results concerning the relation between the period of  a cuspidal  Hecke eigenform $g$ for an elliptic modular group $\varGamma \subseteq SL_2({\bf Z})$  and that of its lift $\widehat g.$  Here, by a lift of $g$ we mean a cuspidal Hecke eigenform for another modular group $\varGamma'$ (e.g. the symplectic group, the orthogonal group, the unitary group, etc.) whose certain $L$-function can be expressed in terms of certain $L$-functions related with $g.$ Thus we propose the following problem: 

\begin{prob}
Let $\langle \widehat g,  \widehat g \rangle $ (resp. $\langle  g,  g \rangle$) be the period of $\widehat g$ (resp.  $g$).
Then express the ratio 
$\displaystyle {\langle \widehat g, \widehat g \rangle / \langle  g,  g \rangle^e }$  in terms of arithmetic invariants of $g,$ for example,  the special values of certain $L$-functions  related with $g$ for some integer $e.$
\end{prob} 

For instance, Zagier \cite{Za} solved  
Problem A for the Doi-Naganuma lift $\widehat f$ of a primitive form $f$ of integral weight. Murase and Sugano \cite{M-S} also solved 
Problem A for the Kudla lift $\widehat f$ of a primitive form $f$ of integral weight. In addition, Kohnen and Skoruppa \cite{K-S} solved 
Problem A in the case where 
$\widehat h$ is the Saito-Kurokawa lift of a cuspidal Hecke eigenform $h$ in the Kohnen plus subspace of half-integral weight. 

We should also note that this type of period relation is not only interesting and important in its own right but also plays an important role in arithmetic theory of modular forms. 
For instance, by using Kohnen and Skoruppa's result, Brown \cite{Br} and Katsurada \cite{Kat2} independently proved a modification of Harder's conjecture on 
congruences occurring between Saito-Kurokawa lifts and non-Saito-Kurokawa lifts under mild conditions. Furthermore, by using such 
congruences, Brown constructed a non-trivial element of a certain Bloch-Kato Selmer group. As for this type of result, see also \cite{B-D-S}. We 
note that this type of congruence relation was conjectured by Doi, Hida and Ishii \cite{D-H-I} in the case where $\widehat f$ is the Doi-Naganuma lift of $f.$ 

Now let us explain our main result briefly. 
Let $k$ and $n$ be positive even integers.  Let $h$ be a  cuspidal Hecke eigenform in the Kohnen plus subspace of weight $k-n/2+1/2$ for $\varGamma_0(4),$  and $f$ the primitive form of weight $2k-n$ for $SL_2({\bf Z})$ corresponding to $h$ under the Shimura correspondence. Then Ikeda \cite{Ik1} constructed a cuspidal Hecke eigenform $I_n(h)$ of weight $k$ for $Sp_n({\bf Z})$ whose standard $L$-function can be expressed as 
$\zeta(s)\prod_{i=1}^n L(s+k-i,f),$ where $\zeta(s)$ is Riemann's zeta function
 and  $L(s,f)$ is Hecke's $L$-function of $f.$
The existence of such a Hecke eigenform was conjectured by Duke and Imamo{\=g}lu in their unpublished paper. 
We call $I_n(h)$ the Duke-Imamo{\=g}lu-Ikeda lift of $h$ (or of $f$). (See also the remark after Theorem 2.1.) We note that $I_2(h)$ is nothing but the Saito-Kurokawa lift of $h.$ Then, as a generalization of the result due to Kohnen and Skoruppa, 
Ikeda among others proposed the following remarkable conjecture in \cite{Ik2}: 
 
\begin{minipage}[t]{0.95\textwidth}
{\it The ratio $\displaystyle {\langle I_n(h), I_n(h) \rangle /
 \langle h, h \rangle}$  should be expressed, up to elementary factor,  as 
 $$L(k,f)\zeta(n) \prod_{i=1}^{n/2-1} L(2i+1,f,{\rm Ad})\zeta(2i),$$ where $L(s,f,{\rm Ad})$ is the adjoint $L$-function of $f$}
(cf. Conjecture A). 
\end{minipage}

\noindent
The aim of this paper is to prove the above 
conjecture (cf. Theorem 2.2). 
We note that $I_n(h)$ is not likely to be realized as a theta lift except in the case $n=2$ (cf. Schulze-Pillot \cite{Schul}). Therefore we cannot use a general method for inner product formulas of theta lifts due to Rallis \cite{Ra}. We also note that the conjecture cannot be explained within the framework of motives since there is no principle so far to associate motives with half-integral weight modular forms. Taking these remarks into account, we take an approach based on the classical Rankin-Selberg method to our problem. Namely, the method we use is to give an explicit formula of the Rankin-Selberg series of a certain half-integral weight Siegel modular form  related with $I_n(h)$, and to compute its residue at a pole. We explain it 
more precisely.

First let $\phi_{I_n(h),1}$ be the first coefficient of the Fourier-Jacobi expansion of $I_n(h)$ and 
$\sigma_{n-1}(\phi_{I_n(h),1})$ the cusp form in the generalized Kohnen plus subspace of weight $k-1/2$ 
for   $\varGamma_0^{(n-1)}(4)$ corresponding to $\phi_{I_n(h),1}$ 
 under the Ibukiyama isomorphism $\sigma_{n-1}.$ For the precise definition of the generalized Kohnen plus subspace and the Ibukiyama isomorphism, see Section 3.  Then we have the following Fourier expansion of $\sigma_{n-1}(\phi_{I_n(h),1}):$  
$$\sigma_{n-1}(\phi_{I_n(h),1})=\displaystyle \sum_A c(A) \exp(2\pi \sqrt{-1}{\rm tr}(AZ)),$$
where $A$ runs over all positive definite half-integral matrices of degree $n-1,$ and ${\rm tr}$ denotes the trace of a matrix. Then, in Section 3, we consider the following Rankin-Selberg series $R(s,\sigma_{n-1}(\phi_{I_n(h),1}))$ 
of $\sigma_{n-1}(\phi_{I_n(h),1}):$
\[
R(s,\sigma_{n-1}(\phi_{I_n(h),1}))=\sum_{A} {|c(A)|^2 \over e(A) (\det A)^s},
\]
where $A$ runs over all the $SL_{n-1}({\bf Z})$-equivalence classes of positive definite half-integral matrices of degree $n-1$ and $e(A)$ denotes the order of the unit group of $A$ in $SL_{n-1}({\bf Z}).$ 
In the integral weight Siegel modular form case, the analytic properties of this type of Dirichlet series have been studied by many people (e.g. Kalinin \cite{Kal}). In the half-integral weight Siegel modular form case, similarly to the integral weight case, we also get 
analytic properties of $R(s,\sigma_{n-1}(\phi_{I_n(h),1})).$ 
While such a 
Dirichlet series with 
no Euler product 
has never 
been regarded as significant as 
automorphic $L$-functions until now
, it should be emphasized that it plays a  
very important role in the proof of our main result. 
Indeed, as one of the most significant
properties, $R(s,\sigma_{n-1}(\phi_{I_n(h),1}))$ has a simple pole at $s=k-1/2$ with residue expressed in terms of the period of $\phi_{I_n(h),1}$ (cf. Corollary to Proposition 3.1). 
Hence, by virtue of the main identity in \cite{K-K1}, this enables us to rewrite Ikeda's conjecture in terms of 
 the relation between the residue of $R(s,\sigma_{n-1}(\phi_{I_n(h),1}))$ at $s=k-1/2$ and the period of $h$ (cf. Theorem 3.2). In order to prove Theorem 3.2, we have to get an explicit formula of  $R(s,\sigma_{n-1}(\phi_{I_n(h),1}))$  in terms of $L(s,f,{\rm Ad})$ and $L(s,f).$ To get it, in Section 4, we reduce our computation to  that of certain formal power series, which  we call formal power series of Rankin-Selberg type, associated with local Siegel series similarly to \cite{I-K2} and \cite{I-K3} (cf. Theorem 4.2). Section 5 is devoted to the computation of them. This computation is similar to those in \cite{I-K2} and \cite{I-K3}, but is more elaborate and longer than them.  In particular we should be careful in dealing with the case $p=2.$ After overcoming such obstacles we can get  explicit formulas of formal power series of Rankin-Selberg type (cf. Theorem 5.3.1). In Section 6, by using Theorem 5.3.1, we immediately get an explicit formula of $R(s,\sigma_{n-1}(\phi_{I_n(h),1}))$ (cf. Theorem 6.2), and by taking the residue of it at $k-1/2$ we prove Theorem 3.2, and therefore prove Conjecture A (cf. Theorem 6.3).  

We note that we can also give an explicit formula of the Rankin-Selberg series of $I_n(h).$ However, it seems difficult to prove Conjecture A directly from such a formula. \vspace*{1mm}

By Theorem 2.2,  we can give a refinement of a result concerning the algebraicity of $\displaystyle { \langle f,  f \rangle^{n/2}/ \langle I_n(h), I_n(h) \rangle}$ due to Choie and Kohnen (cf. Theorem 2.3). Moreover we can apply this result  to characterize prime ideals giving congruences between Duke-Imamo{\=g}lu-Ikeda lifts and non-Duke-Imamo{\=g}lu-Ikeda lifts. This will be discussed in \cite{Kat3}. \vspace*{2mm}

\begin{acknow}
The authors thank   T. Ikeda, Y. Ishikawa, Y. Mizuno and  S. Yamana for their valuable comments.
 The first named author was partly supported by JSPS KAKENHI  Grant Number 24540005, JSPS, and the second named author was partly supported by the JSPS International Training Program (ITP).\vspace*{2mm}
\end{acknow}

\noindent
{\bf Notation.}  
Let $R$ be a commutative ring. We denote by $R^{\times}$ and $R^*$  the semigroup of non-zero elements of $R$ and the unit group of $R,$  respectively. We also put $S^{\Box}=\{a^2 \ | \ a \in S \}$ for a subset $S$ of  $R.$ We denote by $M_{mn}(R)$ the set of
$m \times n$-matrices with entries in $R.$ In particular put $M_n(R)=M_{nn}(R).$   Put $GL_m(R) = \{A \in M_m(R) \ | \ \det A \in R^* \},$ where $\det
A$ denotes the determinant of a square matrix $A$. For an $m \times n$-matrix $X$ and an $m \times m$-matrix
$A$, we write $A[X] = {}^t X A X,$ where $^t X$ denotes the
transpose of $X$. Let $S_n(R)$ denote
the set of symmetric matrices of degree $n$ with entries in
$R.$ Furthermore, if $R$ is an integral domain of characteristic different from $2,$ let  ${\mathcal L}_n(R)$ denote the set of half-integral matrices of degree $n$ over $R$, that is, ${\mathcal L}_n(R)$ is the subset of symmetric
matrices of degree $n$ with entries in the field of fractions of $R$ whose $(i,j)$-component belongs to
$R$ or ${1 \over 2}R$ according as $i=j$ or not.  
In particular, we put ${\mathcal L}_n={\mathcal L}_n({\bf Z})$, and ${\mathcal L}_{n,p}={\mathcal L}_n({\bf Z}_p)$ for a prime number $p.$ 
  For a subset $S$ of $M_n(R)$ we denote by $S^{\times}$ the subset of $S$
consisting of non-degenerate matrices. If $S$ is a subset of $S_n({\bf R})$ with ${\bf R}$ the field of real numbers, we denote by $S_{>0}$ (resp. $S_{\ge 0}$) the subset of $S$
consisting of positive definite (resp. semi-positive definite) matrices. 
$GL_n(R)$ acts on the set $S_n(R)$ in the following way:
\[
GL_n(R) \times S_n(R) \ni (g,A) \longmapsto {}^tg Ag \in S_n(R).
\]
Let $G$ be a subgroup of $GL_n(R).$ For a $G$-stable subset ${\mathcal B}$ of $S_n(R)$  we denote by ${\mathcal B}/G$ the set of equivalence classes of ${\mathcal B}$ under the action of  $G.$ We sometimes identify ${\mathcal B}/G$ with a complete set of representatives of ${\mathcal B}/G.$ We abbreviate ${\mathcal B}/GL_n(R)$ as ${\mathcal B}/\sim$ if there is no fear of confusion. Let $R'$ be a subring of R. Then two symmetric matrices $A$ and $A'$ with
entries in $R$ are said to be equivalent over $R'$ with each
other and write $A \sim_{R'} A'$ if there is
an element $X$ of $GL_n(R')$ such that $A'=A[X].$ We also write $A \sim A'$ if there is no fear of confusion. 
For square matrices $X$ and $Y$ we write $X \bot Y = \mattwo(X;O;O;Y).$

For an integer $D \in {\bf Z}$ such that $D \equiv 0$ or $\equiv 1 \ {\rm mod} \ 4,$ let ${\textfrak d}_D$ be the discriminant of ${\bf Q}(\sqrt{D}),$ and put ${\textfrak f}_D= \sqrt{{ D \over {\textfrak d}_D}}.$ We call an integer $D$ a fundamental discriminant if it is the discriminant of some quadratic extension of ${\bf Q}$ or $1.$ For a fundamental discriminant $D,$ let $\left({\displaystyle
D \over \displaystyle * }\right)$ be the character corresponding to ${\bf Q}(\sqrt{D})/{\bf Q}.$ Here we make the convention that  $\left({\displaystyle
 D  \over \displaystyle * } \right)=1$ if $D=1.$ 

We put ${\bf e}(x)=\exp(2 \pi \sqrt{-1} x)$ for $x \in {\bf C}.$ For a prime number $p$ we denote by $\nu_p(*)$ the additive valuation of ${\bf Q}_p$ normalized so that $\nu_p(p)=1,$ and by ${\bf e}_p(*)$ the continuous additive character of ${\bf Q}_p$ such that ${\bf e}_p(x)= {\bf e}(x)$ for $x \in {\bf Q}.$

For a non-negative integer $r$ we define a polynomial $\phi_r(x)$ in $x$ by $\phi_r(x)=\prod_{i=1}^r (1-x^i).$ Here we understand that $\phi_0(x)=1.$ 

\section{Ikeda's conjecture on the Period of the Duke-Imamo{\=g}lu-Ikeda lift}

 Put $J_n=\mattwo(O_n;-1_n;1_n;O_n),$ where $1_n$ and $O_n$ denotes the unit matrix and the zero matrix of degree $n$, respectively. 
 Furthermore, put 
$$\varGamma^{(n)}=Sp_n({\bf Z})=\{M \in GL_{2n}({\bf Z})   \ | \  J_n[M]=J_n \}. $$
  Let ${\bf H}_n$ be Siegel's
upper half-space of degree $n$.  Let $l$ be an integer or half integer. For a congruence subgroup $\varGamma$ of $\varGamma^{(n)},$ we denote by ${\textfrak M}_{l}(\varGamma)$  the space of holomorphic modular forms of weight $l$  for  $\varGamma.$   We denote by ${\textfrak S}_{l}(\varGamma)$ the subspace of ${\textfrak M}_{l}(\varGamma)$ consisting of cusp forms. 
 For two holomorphic cusp forms $F$ and $G$ of weight $l$ for  $\varGamma$ we define the Petersson  product $\langle F,G \rangle$ by 
$$\langle F,G \rangle=[\varGamma^{(n)}:\varGamma \{\pm 1_{2n}\}]^{-1}\int_{\varGamma \backslash {\bf H}_n} F(Z)\overline {G(Z)} \det ({\rm Im}(Z))^l d^*Z,$$
where $d^*Z$ denotes the invariant volume element on ${\bf H}_n$ defined as usual.
 We call $\langle F, F \rangle$ the period of $F.$ 
  For a positive integer $N,$ let 
 $$\varGamma_0^{(m)}(N)=\left\{\left. \mattwo(A;B;C;D) \in \varGamma^{(m)} \ \right| \ C \equiv O_m \ {\rm mod} \ N \right\},$$
  and in particular put $\varGamma_0(N)=\varGamma_0^{(1)}(N).$  
Let $p$ be a prime number.  For a non-zero element $a \in {\bf Q}_p$ we put $\chi_p(a)=1,-1,$ or
 $0$ according as ${\bf Q}_p(a^{1/2})={\bf Q}_p, {\bf Q}_p(a^{1/2})$ is
 an unramified quadratic extension of ${\bf Q}_p,$ or ${\bf Q}_p(a^{1/2})$
 is a  ramified quadratic extension of ${\bf Q}_p.$
 We note that  $\chi_p(D)=\left({\displaystyle D \over \displaystyle p }\right)$ if $D$ is a fundamental discriminant.
 For an element  $T$ of ${\mathcal L}_{n,p}^{\times}$ with $n$ even, put
 $\xi_p(T)=\chi_p((-1)^{n/2} \det T).$ 
  Let $T$ be an element of ${\mathcal L}_n^{\times}.$  Then  $(-1)^{n/2} \det (2T) \equiv 0 $ or $\equiv 1 \ {\rm mod} \ 4,$ and  we define ${\textfrak d}_T$ and ${\textfrak f}_T$ as ${\textfrak d}_T ={\textfrak d}_{(-1)^{n/2}\det (2T)}$ and ${\textfrak f}_T ={\textfrak  f}_{(-1)^{n/2}\det (2T)},$ respectively. 
 For an element $T$ of ${\mathcal L}_{n,p}^{\times},$ there exists an element $\widetilde T$ of ${\mathcal L}_n^{\times}$ such that $\widetilde T \sim_{{\bf Z}_p} T.$ We then put ${\textfrak e}_p(T)= \nu_p({\textfrak f}_{\widetilde T}),$ and $[{\textfrak d}_T]={\textfrak d}_{\widetilde T} \ {\rm mod} \ {{\bf Z}_p^*}^{\Box}.$ They do not depend on the choice of $\widetilde T.$  We note that $(-1)^{n/2} \det (2T)$ can be expressed as $(-1)^{n/2} \det (2T) =dp^{2{\textfrak e}_p(T)} \ {\rm mod} \ {{\bf Z}_p^*}^{\Box}$ for any $d \in [{\textfrak d}_T].$ 

For each $T \in {\mathcal L}_{n,p}^{\times}$ we define the local Siegel series $b_p(T,s)$ 
 by 
  $$b_p(T,s)=\sum_{R \in S_n({\bf Q}_p)/S_n({\bf Z}_p)} {\bf e}_p({\rm tr}(TR))p^{-\nu_p(\mu_p(R))s},$$  
 where $\mu_p(R)=[R{\bf Z}_p^n+{\bf Z}_p^n:{\bf Z}_p^n].$ 
 We remark that there exists a unique polynomial 
 $F_p(T,X)$ in $X$ such that 
 $$b_p(T,s)=F_p(T,p^{-s}){(1-p^{-s})\prod_{i=1}^{n/2} (1-p^{2i-2s}) \over 1-\xi_p(T)p^{n/2-s}}$$ 
(cf. Kitaoka \cite{Ki1}). 
We then define a polynomial $\widetilde F_p(T,X)$ in $X$ and $X^{-1}$ as
$$\widetilde F_p(B,X)=X^{-{\textfrak e}_p(T)}F_p(T,p^{-(n+1)/2}X).$$
We remark that  $\widetilde{F}_p(B,X^{-1})=\widetilde{F}_p(B,X)$ (cf. \cite{Kat1}). \vspace*{2mm}

Now let $k$ be a positive even integer. Let 
 $$h(z)=\sum_{m \in {\bf Z}_{>0} \atop (-1)^{n/2}m \equiv 0, 1 \ {\rm mod} \ 4 }c_h(m){\bf e}(mz)$$
  be a Hecke eigenform in the Kohnen plus subspace ${\textfrak S}_{k-n/2+1/2}^+(\varGamma_0(4))$ and 
$$f(z)=\sum_{m=1}^{\infty}c_f(m){\bf e}(mz)$$
 the   primitive form in ${\textfrak S}_{2k-n}(\varGamma^{(1)})$ corresponding to $h$ under the  Shimura correspondence (cf. Kohnen,  \cite{Ko}).  For the precise definition of the Kohnen plus subspace, we give it in Section 3 in more general setting.   Let $\alpha_p \in {\bf C}$ such that $\alpha_p+\alpha_p^{-1}=p^{-k+n/2+1/2}c_f(p),$ which we call the Satake $p$-parameter of $f$. Then for a Dirichlet character $\chi$ we  define Hecke's $L$-function $L(s,f,\chi)$ twisted by $\chi$ as 
  $$L(s,f,\chi)=\prod_p \{(1-\alpha_p p^{-s+k-n/2-1/2}\chi(p))(1-\alpha_p^{-1} p^{-s+k-n/2-1/2} \chi(p))\}^{-1}.$$
 In particular, if $\chi$ is the principal character we write $L(s,f,\chi)$ as $L(s,f)$ as usual. 
We define a Fourier series $I_n(h)(Z)$ in $Z \in {\bf H}_n$ by
$$I_n(h)(Z)= \sum_{T \in ({\mathcal L}_n)_{> 0}} c_{I_n(h)}(T){\bf e}({\rm tr}(TZ)),$$
 where
 $$c_{I_n(h)}(T)=c_h(|{\textfrak d}_T|) {\textfrak f}_T^{k-n/2-1/2} \prod_p\widetilde F_p(T,\alpha_p).$$ Then Ikeda \cite{Ik1} showed  the following: 
 
 \begin{thm}
$I_n(h)(Z)$ is a  Hecke eigenform  in ${\textfrak S}_k(\varGamma^{(n)})$ whose 
standard $L$-function coincides with   
$$\zeta(s)\prod_{i=1}^n L(s+k-i,f).$$
\end{thm}

\begin{xrem}  We call $I_n(h)$ the Duke-Imamo{\=g}lu-Ikeda lift of $h$ (or of $f$) as in Section 1. We note that $I_n(h)$ is uniquely determined by $h.$ 
 We also note that $I_2(h)$ coincides with the Saito-Kurokawa lift of $h.$
 Originally, starting from a primitive form $g$ in ${\textfrak S}_{2k-n}(\varGamma^{(1)}),$ Ikeda constructed the $I_n(\widetilde g),$ where $\widetilde g$ is a Hecke eigenform in ${\textfrak S}_{k-n/2+1/2}^+(\varGamma_0(4))$ corresponding to $g$ under the Shimura correspondence. We note that $\widetilde g$ is uniquely determined, only up to constant multiple, by $g$, and therefore so is $I_n(\widetilde g).$
 \end{xrem}

To formulate Ikeda's conjecture, put 
\begin{center}
$\Gamma_{{\bf R}}(s)=\pi^{-s/2}\Gamma(s/2)$\hspace{2mm} and \hspace{2mm}$\Gamma_{{\bf C}}(s)=\Gamma_{{\bf R}}(s)\Gamma_{{\bf R}}(s+1)$. 
\end{center}
We note that
$\Gamma_{{\bf C}}(s)=2(2\pi)^{-s} \Gamma(s)$.
Furthermore put
\begin{center}
$\xi(s)= \Gamma_{{\bf R}}(s)\zeta(s)$\hspace{2mm} and \hspace{2mm}$\widetilde{\xi}(s) = \Gamma_{{\bf C}}(s) \zeta(s)$. 
\end{center}
For a Dirichlet character $\chi$ put 
$$\Lambda(s,\, f,\chi)= {\Gamma_{{\bf C}}(s) L(s,\, f,\chi) \over \tau(\chi)},$$where $\tau(\chi)$ is the Gauss sum of $\chi.$
In particular, we simply write $\Lambda(s,\, f,\chi)$ as $\Lambda(s,\, f)$ if $\chi$ is the principal character. 
 Furthermore,  we define the adjoint $L$-function $L(s,\, f,\, {\rm Ad})$  as 
$$L(s,f,{\rm Ad})=\prod_p \{(1-\alpha_p^2 p^{-s})(1-\alpha_p^{-2} p^{-s})(1-p^{-s}) \}^{-1},$$
 and put 
 $${\Lambda}(s,\, f,\, {\rm Ad})= \Gamma_{{\bf R}}(s+1)\Gamma_{{\bf C}}(s+2k-n-1)L(s,\, f,\, {\rm Ad}),$$
and
$$\widetilde{\Lambda}(s,\, f,\, {\rm Ad})= \Gamma_{{\bf C}}(s)\Gamma_{{\bf C}}(s+2k-n-1)L(s,\, f,\, {\rm Ad}).$$
We note that 
$${\Lambda}(1-s,\, f,\, {\rm Ad})={\Lambda}(s,\, f,\, {\rm Ad}).$$
Then Ikeda \cite{Ik2} among others proposed the following conjecture:
 
 
\begin{conj}
We have  
$$\frac{ \langle I_n(h),\, I_n(h) \rangle}{ \langle h,\, h \rangle }=2^{\alpha(n,\,k)} \Lambda(k,\, f) \widetilde{\xi}(n) \displaystyle\prod_{i=1}^{n/2-1}\widetilde \Lambda(2i+1,\, f,\, {\rm Ad})\, \widetilde{\xi}(2i),$$
where $\alpha(n,k)=-(n-3)(k-n/2)-n+1.$ \vspace*{2mm}
\end{conj}


\begin{xrem}
The primitive form $f$ as well as $I_n(h)$ is uniquely determined by $ h.$ Therefore there is no ambiguity in the above formulation. 
Conjecture A is compatible with the period formula for the Saito-Kurokawa lift proved  by Kohnen and  Skoruppa \cite{K-S} (see also Oda \cite{O}). 
In \cite{Ik2}, Ikeda proposed a more general conjecture for the period of the Ikeda-Miyawaki lift. We also remark that he constructed a lifting from   an elliptic modular form to the space of Hermitian modular forms, and proposed a conjecture similar to the above (cf. \cite{Ik3}).   
\end{xrem}

Now our main result in this paper is the following:

\begin{thm}
Conjecture A holds true for any positive even integer $n.$
\end{thm}

By the above theorem, we obtain the following result:

\begin{thm}
Let the notation be as above. Let $D$ be a
fundamental discriminant such that $(-1)^{n/2}D >0.$ For $i=1,...,n/2-1,$ put 
\\${\bf L}(2i+1,\, f,\,{\rm Ad})=\displaystyle {\widetilde{\Lambda}(2i+1,\, f,\, {\rm Ad}) \over \langle  f,\, f \rangle }.$ Then  
 $$\frac{ |c_h(|D|)|^2 \langle  f,\, f \rangle ^{n/2}}{ \langle I_n(h),\, I_n(h) \rangle} ={\sqrt{-1}
 ^{a_{n}}\, |D|^{k-n/2}\Lambda(k-n/2,f,({D \over *})) \over  2^{b_{n,k}} \Lambda(k,\, f) \widetilde{\xi}(n) \prod_{i=1}^{n/2-1} {\bf L}(2i+1,\, f,\, {\rm Ad}) \widetilde{\xi}(2i)},$$
where $a_{n}=0$ or $1$ according as $n \equiv 0 \ {\rm mod} \ 4$ or $n \equiv \ 2 \ {\rm mod} \ 4,$ and $b_{n,k}$ is some integer depending only on $n$ and $k.$ 
\end{thm}

\begin{proof} By Theorem 1 in \cite{K-Z}
, for any such $D$ 
we have
$${|c_h(|D|)|^2 \over \langle  h,  h \rangle} ={\sqrt{-1}
^{a_{n}}\, 2^{k-n/2-1}|D|^{k-n/2} \Lambda(k-n/2,f,({D \over  *})) \over \langle f, f \rangle}.$$
Thus, by Theorem 2.2, the assertion holds. 
\end{proof}

It is well-known that the value $\displaystyle { \Lambda(k-n/2,f,({D \over *})) \over \sqrt{-1}^{n/2}\Lambda(k,\, f)} $   and the values 
${\bf L}(2i+1,\, f,\,{\rm Ad})$ for $i=1,...,n/2-1$ are algebraic numbers and belong to the Hecke field ${\bf Q}(f)$ if $k >n$ (cf. Shimura \cite{Sh1}, \cite{Sh2}). Thus we obtain 


\begin{xcor}
  Assume that $k >n$ and that all the Fourier coefficients of $ h$ belong to ${\bf Q}(f).$ Then the ratio $\displaystyle \frac{ \langle  f,\, f \rangle ^{n/2}}{ \langle I_n(h),\, I_n(h) \rangle}$ belongs to ${\bf Q}(f).$ 
\end{xcor} 


We note that we can multiply some non-zero complex number $c$ with $h$ so that all the Fourier coefficients of $c h$ belong to ${\bf Q}(f).$  
We also note that the above result has been proved by Furusawa \cite{F} in case $n=2,$ and by Choie and Kohnen \cite{C-K} under the assumption $k >2n$ in general case. Thus Theorem 2.3 and its corollary can be regarded as a refinement of their results. 

\section{Rankin-Selberg series of 
the image of the first Fourier-Jacobi coefficient of 
the Duke-Imamo{\=g}lu-Ikeda lift under the Ibukiyama isomorphism}

To prove Conjecture A, we rewrite it in terms of the residue of the Rankin-Selberg series of a certain half-integral weight Siegel modular form.  Let $l$ be a positive integer. Let $F(Z)$ be an element of ${\textfrak S}_{l-1/2}(\varGamma_0^{(m)}(4)).$ Then $F(Z)$ has the following Fourier expansion:
$$F(Z)= \sum_{A \in ({\mathcal L}_{m})_{>0}} c_F(A) {\bf e}({\rm tr}(AZ))$$
We define the Rankin-Selberg series $R(s,F)$ of $F$ as 
$$R(s,F)=\sum_{A \in ({\mathcal L}_{m})_{>0}/SL_{m}({\bf Z})} {|c_F(A)|^2 \over e(A) (\det A)^s},$$ 
where $e(A)=\#\{X \in SL_{m}({\bf Z}) \ | \ A[X]=A \}.$ 

Put
$${\mathcal L}'_{m}=\{A \in {\mathcal L}_{m}\ | \ A \equiv -\ ^t {r}r  \ {\rm mod} \ 4{\mathcal L}_{m} \ { \rm for \ some } \ r \in {\bf Z}^{m} \}.$$
For $A \in {\mathcal L}'_{m},$ the integral vector $r \in {\bf Z}^{m}$ in the above definition is uniquely determined modulo $2{\bf Z}^{m}$ by $A,$ and is denoted by $r_A.$ Moreover it is easily shown that 
the matrix $\smallmattwo(1;r_A/2;{}^tr_A/2;{({}^tr_Ar_A+A)/4}),$ which will be denoted by $A^{(1)},$ belongs to ${\mathcal L}_{m+1},$ and that its $SL_{m+1}({\bf Z})$-equivalence class is uniquely determined by $A.$ In particular, if $m$ is odd and $A \in ({\mathcal L}'_{m})^{\times},$ put  ${\textfrak d}_A^{(1)}={\textfrak d}_{A^{(1)}},$ and ${\textfrak f}_A^{(1)}={\textfrak f}_{A^{(1)}}.$ Now we define the generalized Kohnen plus subspace of weight $l-1/2$ for  $\varGamma_0^{(m)}(4)$ as 
\begin{eqnarray*}
\lefteqn{
{\textfrak S}_{l-1/2}^{+}(\varGamma_0^{(m)}(4))=
} \\
&{}&\hspace*{-5mm}\left\{ F(Z)=\!\sum_{A \in ({\mathcal L}_{m})_{>0}} \!\!c(A) {\bf e}({\rm tr}(AZ)) \in {\textfrak S}_{l-1/2}(\varGamma_0^{(m)}(4)) \left|  
\begin{array}{ll}
c(A)=0 \\[1mm]
\ {\rm unless} \ A \in (-1)^l{\mathcal L}'_{m}
\end{array}
\right. \!
\right\}. 
\end{eqnarray*} 
Now, for the rest of this section, suppose that $l$ is a positive even integer.  Then there exists an isomorphism from  the space of Jacobi forms of index $1$ to  the generalized Kohnen plus space due to Ibukiyama. To explain this, let $\varGamma_J^{(m)} = \varGamma^{(m)} \ltimes {\bf {\rm H}}_{m}({\bf Z})$, 
where ${\bf {\rm H}}_m({\bf Z})$ is the subgroup of the Heisenberg group ${\bf {\rm H}}_m({\bf R})$ consisting of all elements with integral entries.

Let $J_{l,\, N}^{{\rm cusp}}(\varGamma_J^{(m)})$ denote the space of Jacobi cusp forms of weight $l$ and index $N$ for  the Jacobi group $\varGamma_J^{(m)}.$ 
Let $\phi(Z,z) \in {J_{l,\, 1}^{\, {\rm cusp}}}(\varGamma_J^{(m)}).$ 
Then we  have the following Fourier  expansion:
$${\phi}(Z,\,z)= 
\displaystyle\sum_{\scriptstyle T \in {\mathcal L}_{m},\, r \in {\bf Z}^{m}, \atop {\scriptstyle 4T-{}^t {r}r > 0}}
c(T,r){\bf e}({\rm tr}(T Z)+r^t{z}). 
$$
We say that two elements $(T,r)$ and $(T',r')$ of ${\mathcal L}_{m} \times {\bf Z}^{m}$ are $SL_{m}({\bf Z})$-equivalent  and  write $(T,r) \sim (T',r')$ if there exists an element $g \in SL_{m}({\bf Z})$ such that $T'-{}^tr'r'/4=(T-{}^trr/4)[g].$
We then define a Dirichlet series $R(s,\phi)$ as
$$R(s,\phi)=\sum_{(T,r)} {|c(T,r)|^2 \over e(T-{}^trr/4) (\det (T-{}^trr/4))^s },$$
where $(T,r)$ runs over a complete set of representatives of  $SL_{m}({\bf Z})$-equivalence classes of 
${\mathcal L}_{m} \times {\bf Z}^{m}$ such that $T-{}^trr/4 \in ({\mathcal L}_m)_{ >0}.$
Now $\phi(Z,z)$ can also be expressed as follows:
$$\phi(Z,z)=\sum_{r \in {\bf Z}^{m}/2{\bf Z}^{m}} h_r(Z)\theta_r(Z,z),$$
where $h_r(Z)$ is a holomorphic function on ${\bf H}_{m},$ and 
$$\theta_r(Z,z)=\sum_{\lambda \in  M_{1,m}({\bf Z})} {\bf e}({\rm tr}(Z[{}^t(\lambda+2^{-1}r)])+2(\lambda+2^{-1}r){}^tz).$$
 We note that $h_r(Z)$ have the following Fourier expansion:
$$h_r(Z)=\sum_{T} c(T,r){\bf e}({\rm tr}((T-{}^trr/4)Z)),$$
where $T$ runs over all elements of ${\mathcal L}_{m}$ such that $T-{}^trr/4$ is positive definite. 
Put ${\bf h}(Z)=(h_r(Z))_{r \in {\bf Z}^{m}/2{\bf Z}^{m}}.$ Then ${\bf h}$ is a vector valued modular form of weight $l-1/2$ for  $\varGamma^{(m)},$  that is, for each $\gamma =\smallmattwo(A;B;C;D) \in 
\varGamma^{(m)}$ we have
$${\bf h}(\gamma(Z))=J(\gamma,Z){\bf h}(Z).$$
Here $J(\gamma,Z)$ is an $m \times m$ matrix whose entries are  holomorphic functions on ${\bf H}_{m}$ such that $\overline{{}^tJ(\gamma,Z)}J(\gamma,Z)=|j(\gamma,Z)|^{2l-1}1_{m},$ where $j(\gamma,Z)=\det (CZ+D).$ 
In particular, we have
$$\sum_{r \in {\bf Z}^{m}/2{\bf Z}^{m}} h_r(\gamma(Z))\overline{h_r(\gamma(Z))}
=|j(\gamma,Z)|^{2l-1}\sum_{r \in {\bf Z}^{m}/2{\bf Z}^{m}} h_r(Z)\overline{h_r(Z)}.$$
We then put 
$$\sigma_{m}(\phi)(Z)=\sum_{r \in {\bf Z}^{m}/2{\bf Z}^{m}} h_r(4Z).$$ 
Then Ibukiyama \cite{Ib} showed the following: 
\begin{center}
\begin{minipage}[t]{0.9\textwidth}
{\it  The mapping $\sigma_{m}$ gives a  ${\bf C}$-linear isomorphism 
$$\sigma_{m}:{J_{l,\, 1}^{\, {\rm cusp}}}(\varGamma_J^{(m)}) \simeq
 {\textfrak S}_{l-1/2}^{+}(\varGamma_0^{(m)}(4)),$$
which is compatible with the actions of Hecke operators. }
\end{minipage}
\end{center}
We call $\sigma_{m}$ the Ibukiyama isomorphism. We note that 
$$\sigma_{m}(\phi)=\sum_{A \in  ({\mathcal L}_{m}')_{>0}}c((A+{}^tr_Ar_A)/4,r_A){\bf e}({\rm tr}(AZ)),$$
where $r=r_A$ denotes an element of ${\bf Z}^{m}$ such that $A+{}^tr_Ar_A \in 4 {\mathcal L}_{m}.$ This $r_A$ is uniquely determined up to modulo $2{\bf Z}^{m},$ and $c((A+{}^tr_Ar_A)/4,r_A)$ does not depend on the choice of the representative of $r_A \ {\rm mod} \ 2{\bf Z}^{m}.$
Furthermore, we have  
$$R(s,\sigma_{m}(\phi))= \sum_{A \in  ({\mathcal L}_{m}')_{>0}/SL_{m}({\bf Z})}{ |c((A+\ {}^trr)/4,r)|^2  \over e(A) \det A^{s}},$$
and hence 
$$R(s,\phi)=2^{2sm}R(s,\sigma_{m}(\phi)).$$
Now for  $\phi,\, \psi \in J_{l,\, 1}^{\rm cusp}(\varGamma_J^{(m)})$ 
we define the Petersson  product of $\phi$ and $\psi$ by
$$\langle \phi,\, \psi \rangle = \displaystyle\int_{\varGamma_J^{(m)} \backslash (\H_{m} \times \C^{m})} \!\!\!\!\phi(Z,\, z) \overline{\psi(Z,\, z)} \det(v)^{l-m-2} \exp(-4 \pi  v^{-1} [^t y]) \, dudvdxdy, $$
where $Z = u + \sqrt{-1} v \in \H_{m},\, z = x + \sqrt{-1} y \in {\bf C}^{m}$. 
 Now we consider the analytic properties of $R(s,\phi).$ 

 
\begin{prop}
Let $\phi(Z,z) \in {J_{l,\, 1}^{\, {\rm cusp}}}(\varGamma_J^{(m)}).$  Put 
$${\mathcal R}(s,\phi)=\gamma_{m}(s)\xi(2s+m+2-2l)\prod_{i=1}^{[m/2]} \xi(4s+2m+4-4l-2i)R(s,\phi),$$
where
$$\gamma_{m}(s)=2^{1-2sm}\prod_{i=1}^{m}\Gamma_{\bf R}(2s-i+1).$$
Then the following assertions hold: 
\begin{enumerate}
\item[{\rm (1)}] 
${\mathcal R}(s,\phi)$ has a meromorphic continuation to the whole $s$-plane, and  has the following  functional equation:
$${\mathcal R}(2l-3/2-m/2-s,\phi)={\mathcal R}(s,\phi).$$
\item[{\rm (2)}] ${\mathcal R}(s,\phi)$ is holomorphic for ${\rm Re}(s)>l-1/2,$  and has a simple pole at $s=l-1/2$ with the residue $2^{m+1}\prod_{i=1}^{[m/2]}\xi(2i+1)\langle \phi,\phi \rangle.$
\end{enumerate}


\end{prop}

\begin{proof}
The assertions can be proved in the same manner as in Kalinin \cite{Kal}, but for the  convenience of readers we here give an outline of the  proof. We define the non-holomorphic Siegel Eisenstein series $E^{(m)}(Z,s)$ by 
$$E^{(m)}(Z,s)=(\det {\rm Im}(Z))^{s}\sum_{M \in \varGamma^{(m)}_{\infty} \backslash \varGamma^{(m)}} |j(M,Z)|^{-2s},$$
where
$\varGamma^{(m)}_{\infty}=\left\{\mattwo(A;B;O_m;D) \in \varGamma^{(m)}\right\}.$
For the $\phi(Z,z)$ let ${\bf h}(Z)=(h_r(Z))_{r \in {\bf Z}^{m}/2{\bf Z}^{m}}$ be as above.  Since ${\bf h}$ is a vector valued modular form for  $\Gamma^{(m)},$ we can apply the Rankin-Selberg method and we obtain
$${\mathcal R}(s,\phi)=\int_{\Gamma^{(m)} \backslash {\bf H}_{m}} \sum_{r \in {\bf Z}^{m}/2{\bf Z}^{m}} h_r(Z) \overline{h_r(Z)} {\rm Im}(Z)^{l-1/2} {\mathcal E}^{(m)}(Z,s) d^*Z,$$
where
\begin{eqnarray*}\lefteqn{
{\mathcal E}^{(m)}(Z,s)=\xi(2s+m+2-2l)} \\
&\times&\prod_{i=1}^{[m/2]} \xi(4s+2m+4-4l-2i)E^{(m)}(Z,s+m/2+1-l). 
\end{eqnarray*}
It is well-known that ${\mathcal E}^{(m)}(Z,s)$ has a meromorphic continuation to the whole $s$-plane, and has the following functional equation:

$${\mathcal E}^{(m)}(Z,2l-3/2-m/2-s)={\mathcal E}^{(m)}(Z,s).$$ 
Thus the first assertion (1) holds. Furthermore it is holomorphic for ${\rm Re}(s) > l-1/2,$ and has a simple pole at $s=l-1/2$ with the residue
$\prod_{i=1}^{[m/2]} \xi(2j+1).$
We note that 
$$\langle \phi,\, \phi \rangle =2^{-m-1}   \int_{\varGamma^{(m)} \backslash {\bf H}_{m}}  \sum_{r \in {\bf Z}^{m}/2{\bf Z}^{m}} h_r(Z)\overline{h_r(Z)} {\rm Im}(Z)^{l-1/2} d^*Z.$$
Thus the second assertion (2) holds.
\end{proof}

 For $F \in {\textfrak S}_{l-1/2}^+(\varGamma_0^{(m)}(4))$ put
\begin{eqnarray*}
\lefteqn{
{\mathcal R}(s,F)=\prod_{i=1}^{m}\Gamma_{\bf R}(2s-i+1) } \\
&\times& \xi(2s+m+2-2l)\prod_{i=1}^{[m/2]} \xi(4s+2m+4-4l-2i)R(s,F).
\end{eqnarray*}
We note that 
$${\mathcal R}(s,\sigma_{m}(\phi))=2^{-1}{\mathcal R}(s,\phi)$$
for $\phi \in {J_{l,\, 1}^{\, {\rm cusp}}}(\varGamma_J^{(m)}).$
Thus we obtain


\begin{xcor}
Let the notation and the assumption be as Proposition 3.1. Then ${\mathcal R}(s,\sigma_{m}(\phi))$ has a meromorphic continuation to the whole $s$-plane, and  has the following  functional equation:
$${\mathcal R}(2l-3/2-m/2-s,\sigma_{m}(\phi))={\mathcal R}(s,\sigma_{m}(\phi)).$$
Furthermore it is holomorphic for ${\rm Re}(s) > l-1/2,$ and has a simple pole at $s=l-1/2$ with the residue $2^{m}\prod_{i=1}^{[m/2]}\xi(2i+1)\langle \phi,\phi \rangle.$
\end{xcor}

Let $n$ and $k$ be positive even integers. Let $h$ be a Hecke eigenform in ${\textfrak S}_{k-n/2+1/2}^+(\varGamma_0(4)),$ and $f$ and $I_n(h)$ be as in Section 2. Write $Z \in {\bf H}_n$ as $Z=\left(\begin{array}{cc}
  \tau' & z  \\
  {}^t{z} & \tau  \\
\end{array}
\right)$
with $\tau \in {\bf H}_{n-1},\, z \in \C^{n-1}$ and $\tau' \in {\bf H}_1$. Then we have the following Fourier-Jacobi expansion of $I_n(h)$:
$$
I_n(h)\!\left( 
\left(
\begin{array}{cc}
  \tau' & z  \\
  {}^t{z} & \tau  \\
\end{array}
\right)
\right)
=\sum_{N=0}^{\infty} \phi_{I_n(h),N}(\tau,\,z) {\bf e}(N\tau'), 
$$
where $\phi_{I_n(h),N}(\tau,\,z)$ is called the $N$-th Fourier-Jacobi coefficient of $I_n(h)$ and defined by 
$$
\phi_{I_n(h),N}(\tau,\,z) = 
\sum_{T \in {\mathcal L}_{n-1},\, r \in {\bf Z}^{n-1}, \atop { 4NT-{}^t{r}r > 0}}
c_{I_n(h)}
\left(
\left(
\begin{array}{cc}
  N & r/2  \\
  {}^t{r}/2 & T  
\end{array}
\right)
\right)
{\bf e}({\rm tr}(T \tau) + r\, {}^t{z}). 
$$
We easily see that $\phi_{I_n(h),N}$ belongs to $J_{k,\, N}^{\, {\rm cusp}}(\varGamma_J^{(n-1)}) $ for each $N \in {\bf Z}_{>0}$. 
{


Under the above notation, we will prove the following theorem in Section 6:

%

\begin{thm}
\begin{eqnarray*}
\lefteqn{
{\rm Res}_{s=k-1/2} {\mathcal R}(s,\sigma_{n-1}(\phi_{I_n(h),1}))
} \\
&=& 2^{\beta(n,\,k)}\langle  h,\,  h \rangle \prod_{i=1}^{n/2 - 1}  \widetilde{\xi}(2i) \xi(2i+1)\widetilde{\Lambda}(2i+1,\, f,\, {\rm Ad}),
\end{eqnarray*}
where $\beta(n,k)=-(n-4)k+(n^2-5n+2)/2.$
\end{thm}

Then we can show the following:


\begin{thm}
Under the above notation and the assumption, Theorem 3.2 implies Conjecture A.
\end{thm}

\begin{proof} By Corollary to Main Theorem of \cite{K-K1}, we have
 $$\frac{\langle I_n(h),\, I_n(h) \rangle}{\langle \phi_{I_n(h),1},\, \phi_{I_n(h),1} \rangle} = 2^{-k+n-1}\Lambda(k,\, f)\widetilde{\xi}(n)$$
  (see the remark below). 
 Thus Conjecture A holds true if and only if
$$\langle \phi_{I_n(h),1},\, \phi_{I_n(h),1} \rangle=2^{-k(n-4) +n(n-7)/2+2}\langle  h,\,  h \rangle \prod_{i=1}^{n/2-1}  \widetilde{\xi}(2i) \widetilde{\Lambda}(2i+1,\, f,\, {\rm Ad}).$$
On the other hand, by Corollary to Proposition 3.1 we have
$${\rm Res}_{s=k-1/2} {\mathcal R}(s,\sigma_{n-1}(\phi_{I_n(h),1}))=2^{n-1} \langle \phi_{I_n(h),1},\, \phi_{I_n(h),1} \rangle \prod_{i=1}^{n/2-1} \xi(2i+1).$$
Thus the assertion holds.
\end{proof}

\begin{xrem}
In \cite{K-K1}, we incorrectly quoted Yamazaki's result in \cite{Y}. Indeed \lq\lq$\langle F, G \rangle$" on  the page 2026, line 14 of \cite{K-K1}
 should read \lq\lq ${\displaystyle 1 \over \displaystyle 2}\langle F, G \rangle$" (cf. Krieg \cite{Kr}) and therefore \lq\lq $2^{2k-n+1}$" on  the page 2027, line 7 of \cite{K-K1} should read  \lq\lq $2^{2k-n}$".
\end{xrem}

\section{Reduction to local computations}

To prove Theorem 3.2, we give an explicit formula for $R(s,\sigma_{n-1}(\phi_{I_n(h),1}))$ for the first Fourier-Jacobi coefficient $\phi_{I_n(h),1}$ of $I_n(h).$ To do this, we reduce the problem to local computations. 

For $a,b \in {\bf Q}_p^{\times}$ let $(a,b)_p$ the Hilbert symbol on ${\bf Q}_p.$ Following Kitaoka \cite{Ki2}, we define the Hasse invariant $\varepsilon(A)$ of $A \in S_m({\bf Q}_p)^{\times}$ by 
$$\varepsilon(A)=\prod_{1 \le i \le j \le m}(a_i,a_j)_p$$
if $A$ is equivalent to $a_1 \bot \cdots \bot a_m$ over ${\bf Q}_p$ with some $a_1,a_2,...,a_m \in {\bf Q}_p^{\times}.$ We note that this definition does not depend on the choice of $a_1,a_2,...,a_m.$ 

Now put
$${\mathcal L}_{m,p}'=\{ A \in {\mathcal L}_{m,p} \ | \ A \equiv -\ ^t {r}r  \ {\rm mod} \ 4{\mathcal L}_{m,p} \ { \rm for \ some } \ r \in {\bf Z}_p^{m} \} .$$
Furthermore we put  $S_{m}({\bf Z}_p)_e=2{\mathcal L}_{m,p}$ and $S_{m}({\bf Z}_p)_o=S_{m}({\bf Z}_p) \setminus S_{m}({\bf Z}_p)_e$. We note that 
 ${\mathcal L}_{m,p}'={\mathcal L}_{m,p}=S_{m}({\bf Z}_p)$ if $p\not=2.$ Let $T \in {\mathcal L}_{m-1,p}'.$ Then there exists an element $r \in {\bf Z}_p^{m-1}$  such that $\smallmattwo(1;r/2;{}^tr/2;{(T+{}^trr)/4})$ belongs to ${\mathcal L}_{m,p}.$  As is easily shown, $r$ is uniquely determined by $T$, up to modulo $2{\bf Z}_p^{m-1},$ and is denoted by $r_T.$ Moreover as will be shown in the next lemma, $\smallmattwo(1;r_T/2;{}^tr_T/2;{(T+{}^tr_Tr_T)/4})$ is uniquely determined by $T$, up to $GL_m({\bf Z}_p)$-equivalence, and is denoted by $T^{(1)}.$ 


\begin{lem}
{\rm ([\cite{K-K4}, Lemma 3.1])} Let $m$ be a positive integer. 
\begin{enumerate}
\item[{\rm (1)}] 
Let $A$ and $B$ be elements of ${\mathcal L}_{m-1,p}'.$ 
Then $\smallmattwo(1;r_A/2;{}^tr_A/2;{(A+{}^tr_Ar_A)/4}) \sim \smallmattwo(1;r_B/2;{}^tr_B/2;{(B+{}^tr_Br_B)/4})$ if $A \sim B.$
\item[{\rm (2)}] Let $A \in {\mathcal L}_{m-1,p}'.$ 
\begin{enumerate}
\item[{\rm (2.1)}]  Let $p\not=2.$ Then  
$A^{(1)}  \sim \mattwo(1;0;0;A).$

\item[{\rm (2.2)}] {
Let $p=2.$ If  $r_A \equiv 0 \ {\rm mod} \ 2,$ then 
$A \sim 4B$ with $B \in {\mathcal L}_{m-1,2},$ and 
$ A^{(1)}  \sim \mattwo(1;0;0;B).$ In particular, ${\rm ord}((\det B)) \ge m$ or $\ge m+1$ according as $m$ is even or odd.

If  $r_A \not\equiv 0 \ {\rm mod} \ 2,$ then 
$A \sim a \bot 4B$ with $a \equiv -1 \ {\rm mod} \ 4$ and $B \in {\mathcal L}_{m-2,2}$ and  we have
$A^{(1)}  \sim \matthree(1;1/2;0;1/2;{(a+1)/4};0;0;0; B).$ In particular, ${\rm ord}((\det B)) \ge m$ or $\ge m-1$ according as $m$ is even or odd.
}
\end{enumerate}
\end{enumerate}
\end{lem}

%
Now let $m$ be a positive even integer. For $T \in ({\mathcal L}_{m-1,p}')^{\times},$ we define $[{\textfrak d}_T^{(1)}]$ and ${\textfrak e}^{(1)}_T$ as $[{\textfrak d}_{T^{(1)}}]$ and ${\textfrak e}_{T^{(1)}},$ respectively. These do not depend on the choice of $r_T.$ 
We note that  $(-1)^{m/2} \det T= 2^{m-2}d p^{2{\textfrak e}_T^{(1)}} \ {\rm mod} \ {{\bf Z}_p^*}^{\Box}$ for any $d \in [{\textfrak d}_T^{(1)}].$ 
We define a polynomial $F_p^{(1)}(T,X)$ in $X,$ and a polynomial $\widetilde F_p^{(1)}(T,X)$ in $X$ and $X^{-1}$ by 
$$F_p^{(1)}(T,X)=F_p(T^{(1)}, X), $$ 
and 
$$\widetilde F_p^{(1)}(T,X)= X^{-{\textfrak e}_p^{(1)}(T)} F_p^{(1)}(T,p^{-(n+1)/2}X).$$
Let $B$ be  a half-integral matrix $B$ over ${\bf Z}_p$ of degree $m.$ Let $p \not=2.$ Then 
 $$\widetilde F_p^{(1)}(B,X)=\widetilde F_p(1 \bot B,X).$$
 Let $p=2.$ Then  
\[
\widetilde F_2^{(1)}(B,X)=
\left\{
\begin{array}{cl}
\widetilde F_2({\mattwo(1;1/2;1/2;{(a+1)/4})} \bot B',X) &  
{
\begin{array}{ll}
{\rm if} \ B=a \bot 4B' \\
\hspace{0.5cm} {\rm with} \ a \equiv -1 \ {\rm mod} \ 4,
\end{array}
} \\[5mm]
\widetilde F_2(1 \bot B',X) 
& {\rm if} \ B=4B'. 
\end{array}
\right.
\]
Now for each $T \in S_m({\bf Z}_p)_e^{\times}$ put 
$$F_p^{(0)}(T,X)=F_p(2^{-\delta_{2,p}}T,X)$$ and 
$$\widetilde F_p^{(0)}(T,X)=\widetilde F_p(2^{-\delta_{2,p}}T,X).$$
We define $[{\textfrak d}_T^{(0)}]$ and ${\textfrak e}^{(0)}_T$ as $[{\textfrak d}_{2^{-\delta_{2,p}}T}]$ and ${\textfrak e}_{2^{-1}T},$ respectively. We note that  $(-1)^{m/2}\det T= d p^{2{\textfrak e}_T^{(0)}} \ {\rm mod} \ {{\bf Z}_p^*}^{\Box}$ for any $d \in [{\textfrak d}_T^{(0)}].$ 

Now let $m$ and $l$ be positive integers such that $m \ge l.$ Then for  non-degenerate symmetric matrices $A$ and  $B$  of degree $m$ and $l$ respectively with entries in ${\bf Z}_p$  we define the local density $\alpha_p(A,B)$ and the primitive local density $\beta_p(A,B)$ representing $B$ by $A$ as
$$\alpha_p(A,B)=2^{-\delta_{m,l}}\lim_{a \rightarrow
\infty}p^{a(-ml+l(l+1)/2)}\#{\mathcal A}_a(A,B),$$
 and
 $$\beta_p(A,B)=2^{-\delta_{m,l}}\lim_{a \rightarrow
\infty}p^{a(-ml+l(l+1)/2)}\#{\mathcal B}_a(A,B),$$
where $${\mathcal A}_a(A,B)=\{X \in
M_{ml}({\bf Z}_p)/p^aM_{ml}({\bf Z}_p) \ | \ A[X]-B \in p^aS_l({\bf Z}_p)_e \},$$
and 
$${\mathcal B}_a(A,B)=\{X \in {\mathcal A}_a(A,B) \ | \ 
  {\rm rank}_{{\bf Z}_p/p{\bf Z}_p} (X \ {\rm mod} \ p) =l \}.$$
In particular we write $\alpha_p(A)=\alpha_p(A,A).$ 
Furthermore put 
$$M(A)=\sum_{A' \in {\mathcal G}(A)} {1 \over e(A')}$$
for a positive definite symmetric matrix $A$ of degree $n-1$ with entries in ${\bf Z},$ where ${\mathcal G}(A)$ denotes the set of $SL_{n-1}({\bf Z})$-equivalence classes belonging to the genus of $A.$   Then
by Siegel's main theorem on the quadratic forms, we obtain 
$$M(A)=2^{2-n}e_{n-1}\kappa_{n-1} \det A^{n/2} \prod_p \alpha_p(A)^{-1}$$
where $e_{n-1}=1$ or $2$ according as $n=2$ or not, and 
$$\kappa_{n-1}=\prod_{i=1}^{(n-2)/2}\Gamma_{\bf C}(2i)$$
(cf. Theorem 6.8.1 in \cite{Ki2}
).  Put  
$${\mathcal F}_{p}=\{d_0 \in {\bf Z}_p \ | \ \nu_p(d_0) \le 1\}$$ 
 if $p$ is an odd prime, and  
$${\mathcal F}_{2}=\{d_0 \in {\bf Z}_2 \ | \  d_0 \equiv 1 \ {\rm mod} \ 4, \ {\rm  or} \  d_0/4  \equiv -1 \  {\rm mod} \ 4,  \ {\rm or} \ \nu_2(d_0)=3 \}.$$
For $d \in {\bf Z}_p^{\times}$ put 
 \begin{eqnarray*}
\lefteqn{S_{m}({\bf Z}_p,d)}\\
=&\{T \in S_{m}({\bf Z}_p) \ |\ (-1)^{[(m+1)/2]} \det T=p^{2i}d \ {\rm mod} \ {{\bf Z}_p^*}^{\Box} \ {\rm with \ some \ } \ i \in {\bf Z} \}, 
\end{eqnarray*}
and $S_{m}({\bf Z}_p,d)_x=S_{m}({\bf Z}_p,d) \cap S_{m}({\bf Z}_p)_x$ for $x=e$ or $o.$   We note that $S_{m}({\bf Z}_p,d)=S_{m}({\bf Z}_p,p^jd)$ for any even integer $j.$  If $m$ is even, put ${\mathcal L}_{m,p}^{(0)}=S_{m}({\bf Z}_p)_e^{\times}$ and ${\mathcal L}_{m-1,p}^{(1)}=({\mathcal L}_{m-1,p}')^{\times}.$   We also define ${\mathcal L}_{m-j,p}^{(j)}(d)= S_{m-j}({\bf Z}_p,d) \cap {\mathcal L}_{m-j,p}^{(j)}$  for $j=0,1.$ 
Let $m$ be an even integer. For $d_0 \in {\mathcal F}_p,l=0,1$ and $j=0,1,$ define a rational number $\kappa(d_0,m-j,l)=\kappa_p(d_0,m-j,l)$  as
 $$\kappa(d_0,m-j,l)=\left\{\begin{array}{ll}
 \{(-1)^{lm(m-2)/8} 2^{-(m-2)(m-1)/2}\}^{\delta_{2,p}} &  {}\\
 {} \hspace{2mm} \times ((-1)^{m/2},(-1)^{m/2}d_0 )_p^l\,\, p^{-(m/2-1)l\nu(d_0)} & \ {\rm if} \ j=1 \\ 
 { } & {} \\
\{(-1)^{m(m+2)/8}\,((-1)^{m/2}2,\,d_0)_2\}^{l\delta_{2,p}} & \ {\rm if} \  j=0.
\end{array}
\right. $$
Let $\iota_{m,p}$ be the constant function on ${\mathcal L}_{m,p}^{\times}$ taking the value 1, and $\varepsilon_{m,p}$ the function on ${\mathcal L}_{m,p}^{\times}$ assigning the Hasse invariant of $A$ for $A \in {\mathcal L}_{m,p}^{\times}.$  We sometimes drop the suffix and write $\iota_{m,p}$ as $\iota_p$ or $\iota$ and so on if there is no fear of confusion. 
From now on we sometimes write $\omega=\varepsilon^l$ with $l=0$ or $1$ according as $\omega=\iota$ or $\varepsilon.$ Let $n$ be an even integer. For $d_0 \in {\mathcal F}_{p}$ and  $\omega=\varepsilon^l$ with $l=0,1$ we define a formal power series $H_{n-1,p}(d_0,\omega,X,Y,t) \in {\bf C}[X,X^{-1},Y,Y^{-1}][[t]]$ by 
\begin{eqnarray*}
\lefteqn{
H_{n-1,p}(d_0,\omega,X,Y,t)=\kappa(d_0,n-1,l)^{-1}t^{\delta_{2,p}(2-n)}
} \\
&\times & \sum_{A \in {\mathcal L}_{n-1,p}^{(1)}(d_0)/GL_{n-1}({\bf Z}_p)} {\widetilde F_p^{(1)}(A,X)\widetilde F_p^{(1)}(A,Y) \over \alpha_p(A)} \varepsilon(A)^l t^{\nu_p(\det A)}.
\end{eqnarray*}
We call $H_{n-1,p}(d_0,\omega,X,Y,t)$ a formal power series of Rankin-Selberg type. An explicit formula for $H_{n-1,p}(d_0,\omega_p,X,Y,t)$ will be given in the next section.  
Let ${\mathcal F}$ denote the set of fundamental discriminants, and for $l=\pm 1,$ put 
$${\mathcal F}^{(l)}=\{ d_0 \in {\mathcal F} \  | \ ld_0 >0 \}.$$
Now let $h$ be a Hecke eigenform in ${\textfrak S}_{k-n/2+1/2}^+(\varGamma_0(4)),$ and $ f,I_n(h),\phi_{I_n(h),1}$ and $\sigma_{n-1}(\phi_{I_n(h),1})$ be as in Section 3. Then we have 

\begin{thm}
 Let the notation and the assumption be as above.  Then for ${\rm Re}(s) \gg 0,$ we have 
$$R(s,\sigma_{n-1}(\phi_{I_n(h),1})) ={e_{n-1} \over 2}\kappa_{n-1}2^{(-s-1/2)(n-2)}$$
$$ \times \{ \sum_{d_0 \in {\mathcal F}^{((-1)^{n/2})} }|c_h(|d_0|)|^2|d_0|^{n/2-k+1/2} 
\prod_p H_{n-1,p}(d_0,\iota_p,\alpha_p,\alpha_p,p^{-s+k-1/2})  $$
$$+ (-1)^{n(n-2)/8}\sum_{d_0 \in {\mathcal F}^{((-1)^{n/2})} }|c_h(|d_0|)|^2|d_0|^{-k+3/2} \prod_p H_{n-1,p}(d_0,\varepsilon_p,\alpha_p,\alpha_p,p^{-s+k-1/2})\},$$
where $\alpha_p$ is the Satake $p$-parameter of $f$.
\end{thm}
\begin{proof}
Let $T \in ({\mathcal L}'_{n-1})_{>0}.$ Then it follows from Lemma 4.1 that the $T$-th Fourier coefficient $c_{\sigma_{n-1}(\phi_{I_n(h),1})}(T)$  of $\sigma_{n-1}(\phi_{I_n(h),1})$ is uniquely determined by the genus to which $T$ belongs, and, by definition, it can be expressed as
$$c_{\sigma_{n-1}(\phi_{I_n(h),1})}(T)=c_{I_n(h)}(T^{(1)})=c_h(|{\textfrak d}_T^{(1)}|)({\textfrak f}_T^{(1)})^{k-n/2-1/2}\prod_p \widetilde F^{(1)}(T,\alpha_p).$$
We also note that $\prod_p((-1)^{n/2},(-1)^{n/2}d_0)_p =1$ for any $d_0 \in {\mathcal F}^{((-1)^{n/2})},$ and hence
$$\prod_p \kappa_p(d_0,n-1,0)=2^{-(n-2)(n-1)/2}$$
and
$$\prod_p \kappa_p(d_0,n-1,1)=2^{-(n-2)(n-1)/2} (-1)^{n(n-2)/8}|d_0|^{-n/2+1}.$$
 Thus, by using the same method as in Proposition 2.2 of \cite{I-S}, similarly to \cite{I-K1}, Theorem 3.3, (1), and \cite{I-K2}, Theorem 3.2,  we obtain the assertion.
 \end{proof}

\section{Formal power series associated with local Siegel series}

Throughout this section we fix a positive even integer $n.$ 
 In this section we give an explicit formula of $H_{n-1}(d_0,\omega,X,Y,t)=H_{n-1,p}(d_0,\omega,X,Y,t)$  for $\omega=\iota, \varepsilon$ (cf. Theorem 5.3.1). The idea is to rewrite $H_{n-1}(d_0,\omega,X,Y,t)$ in terms of various power series. 
  Henceforth, for a $GL_m({\bf Z}_p)$-stable subset ${\mathcal B}$ of $S_m({\bf Q}_p),$ we simply write  $\sum_{T \in {\mathcal B}}$ instead of $\sum_{T \in {\mathcal B}/\sim}$ if there is no fear of confusion. We also simply write $\nu_p$ as $\nu$ and the others if the prime number $p$ is clear from the context.

\subsection{Formal power series of Andrianov type}

\noindent
  { }
  
  \bigskip

For an $m \times m$ half-integral matrix $B$ over ${\bf Z}_p,$ let $(\overline{W},\overline {q})$ denote the quadratic space over ${\bf Z}_p/p{\bf Z}_p$ defined by the quadratic form $\overline {q}({\bf x})=B[{\bf x}] \ {\rm mod} \ p,$ and define the radical $R(\overline {W})$ of $\overline {W}$ by
$$R(\overline {W})=\{{\bf x} \in \overline {W} \ | \ \overline {B}({\bf x},{\bf y})=0 \ {\rm for \ any } \ {\bf y} \in \overline {W} \},$$
where $\overline {B}$ denotes the associated symmetric bilinear form of $\overline {q}.$ 
We then put $l_p(B)= {\rm rank}_{{\bf Z}_p/p{\bf Z}_p} R(\overline {W})^{\perp},$ where $R(\overline {W})^{\perp}$ is the orthogonal complement of $R(\overline {W})^{\perp}$ in $\overline {W}.$ Furthermore, in case $l_p(B)$ is even,  put $\overline {\xi}_p(B)=1$ or $-1$ according as $R(\overline {W})^{\perp}$ is hyperbolic or not. In case $l_p(B)$ is odd, we put $\overline {\xi}_p(B)=0.$ Here we make the convention that $\xi_p(B)=1$ if $l_p(B)=0.$ We note that $\overline {\xi}_p(B)$ is different from the $\xi_p(B)$ in general, but they coincide if $B \in {\mathcal L}_{m,p} \cap {1 \over 2}GL_m({\bf Z}_p).$  
 
Let $m$ be a positive even integer.
For $B \in {\mathcal L}_{m-1,p}^{(1)}$ put $B^{(1)}=\mattwo(1;r/2; {}^tr/2;{(B+{}^trr)/4}),$ where $r$ is an element of ${\bf Z}_p^{m-1}$ such that $B+{}^trr \in 4{\mathcal L}_{m-1,p}.$ Then we put
$\xi^{(1)}(B)=\xi(B^{(1)})$ and $\overline {\xi}^{(1)}(B)=\overline {\xi}(B^{(1)}).$ These do not depend on the choice of $r,$ and 
we have $\xi^{(1)}(B)=\chi((-1)^{m/2}\det B).$ 
Let $p \not=2.$ Then an element $B$ of ${\mathcal L}_{m-1,p}^{(1)}$ is equivalent, over ${\bf Z}_p$, to $\Theta \bot pB_1$
 with $\Theta \in GL_{m-n_1-1}({\bf Z}_p) \cap S_{m-n_1-1}({\bf Z}_p)$ and $B_1 \in S_{n_1}({\bf Z}_p)^{\times}.$ Then $\overline {\xi}(B)=0$ if $n_1$ is odd, and $\overline {\xi}^{(1)}(B)=\chi((-1)^{(m-n_1)/2} \det \Theta)$ if $n_1$ is even.
Let $p=2.$ Then an element $B \in {\mathcal L}_{m-1,2}^{(1)}$ is equivalent, over ${\bf Z}_2,$  to a matrix of the form $2\Theta \bot B_1,$ where $\Theta \in GL_{m-n_1-2}({\bf Z}_2) \cap S_{m-n_1-2}({\bf Z}_2)_e$ and $B_1$ is one of the following three types: \begin{enumerate}
\item[(I)]
 $B_1=a  \bot 4B_2$ with $a \equiv -1 \ {\rm mod} \ 4, $ and $B_2 \in S_{n_1}({\bf Z}_2)_e^{\times}$; \vspace*{1mm}

\item[(II)]
 $B_1 \in  4S_{n_1+1}({\bf Z}_2)^{\times}$; \vspace*{1mm} 

\item[(III)]
 $B_1=a  \bot 4B_2$ with $a \equiv -1 \ {\rm mod} \ 4, $ and $B_2 \in S_{n_1}({\bf Z}_2)_o.$
 
 \end{enumerate}
Then  $\overline {\xi}^{(1)}(B)=0$ if $B_1$ is of type (II) or type (III). Let $B_1$ be of type (I). Then $(-1)^{(m-n_1)/2} a \det \Theta$  mod $({\bf Z}_2^*)^{\Box}$ is uniquely determined by $B$  and we have
\[\overline {\xi}^{(1)}(B)=\chi((-1)^{(m-n_1)/2} a \det \Theta).\] 
Suppose that $p \neq 2,$ and let ${\mathcal U}={\mathcal U}_p$  be a complete set of representatives of ${\bf Z}_p^*/({\bf Z}_p^*)^{\Box}.$ Then, for each positive integer $l$ and $d \in {\mathcal U}_{p}$, there exists a unique, up to ${\bf Z}_p$-equivalence, element of $S_{l}({\bf Z}_p)  \cap GL_{l}({\bf Z}_p)$ whose determinant is $(-1)^{[(l+1)/2]}d,$ which will be denoted by  $\Theta_{l,d}.$ 
Suppose that $p=2,$ and put  ${\mathcal U}={\mathcal U}_{2}=\{1 ,5 \}.$ Then for each positive even integer $l$ and $d \in {\mathcal U}_{2}$ there exists a unique, up to ${\bf Z}_2$-equivalence, element of $S_{l}({\bf Z}_2)_{e} \cap GL_{l}({\bf Z}_2)$ whose determinant is  $(-1)^{l/2} d,$ which will be also denoted by  $\Theta_{l,d}.$ In particular, if $p$ is any prime number and $l$ is even, we put $\Theta_l=\Theta_{l,1}$ We make the convention that $\Theta_{l,d}$ is the empty matrix if $l=0.$ For an element $d \in {\mathcal U}$ we use the same symbol $d$ to denote the coset $d$ mod  $({\bf Z}_p^*)^{\Box}.$

We put  ${\mathcal D}_{l,i}=GL_n({\bf Z}_p) \mattwo(1_{l-i};0;0;p1_i) GL_l({\bf Z}_p)$ for $0 \le i \le l.$ 
Suppose that $r$ is a positive even integer. For $j=0,1,\xi=\pm 1$ and  $T \in {\mathcal L}_{r-j,p}^{(j)},$  we define a  polynomial $\widetilde F_p^{(j)}(T,\xi,X)$ in $X$ and $X^{-1}$ by
$$\widetilde F_p^{(j)}(T,\xi,X)=X^{-{\textfrak e}^{(j)}(T)}F_p^{(j)}(T,\xi X).$$
We note that $\widetilde F_p^{(j)}(T,\xi,X)=\xi^{e^{(j)}(T)}\widetilde F_p^{(j)}(T,\xi X),$ and in particular 
$\widetilde F_p^{(j)}(T,1,X)$ coincides with $\widetilde F_p^{(j)}(T,X).$ We also define a polynomial $\widetilde G_p^{(j)}(T,\xi,X,t)$ in $X,X^{-1}$ and $t$ by
$$\widetilde G_p^{(j)}(T,\xi,X,t)=\sum_{i=0}^{r-j} (-1)^i p^{i(i-1)/2}t^i \sum_{D \in GL_{r-j}({\bf Z}_p) \backslash {\mathcal D}_{r-j,i}} \widetilde F^{(j)}_p(T[D^{-1}],\xi,X), $$
and put $\widetilde G_p^{(j)}(T,X,t)=\widetilde G_p^{(j)}(T,1,X,t).$
  We also define a polynomial $G_p^{(j)}(T,X)$ in $X$ by
\begin{eqnarray*}
\lefteqn{
G_p^{(j)}(T,X)} \\
&=&\sum_{i=0}^{r-j} (-1)^i p^{i(i-1)/2} (X^2p^{r+1-j})^i \sum_{D \in GL_{r-j}({\bf Z}_p) \backslash {\mathcal D}_{r-j,i}} F_p^{(j)}(T[D^{-1}],X). 
\end{eqnarray*}
We note that
$$\widetilde G_p^{(j)}(T,X,1)= X^{-{\textfrak e}^{(j)}(T)} G_p^{(j)}(T,Xp^{-(m+1)/2}).$$

 Now for an element $T \in {\mathcal L}_{r-1,p}^{(1)}$ we define a polynomial $B_p^{(1)}(T,t)$ in $t$ by  
$$B_p^{(1)}(T,t)={(1-\xi_p(T^{(1)})p^{-r/2+1/2}t)\prod_{i=1}^{(r-2)/2}(1-p^{-2i+1}t^2) \over G_p^{(1)}(T,p^{-r+1/2}t)}.$$
Then by [\cite{K-K4}, Lemma 4.2.1] we have the following:

 
\begin{lems}
Let $n$ be the fixed positive even integer. Let $B \in {\mathcal L}_{n-1,p}^{(1)}.$   \vspace*{1mm}

\noindent
{\rm (1)} Let $p \not=2, $ and suppose that  $B=\Theta_{n-n_1-1,d} \bot pB_1$  with $d \in {\mathcal U}$ and $B_1 \in S_{n_1}({\bf Z}_p)^{\times}.$ 
Then 
\[
B_p^{(1)}(B,t)=
\left\{
\begin{array}{cl}
\displaystyle (1-\overline{\xi}^{(1)}(B) p^{(n_1-n+1)/2}t)\prod_{i=1}^{(n-n_1-2)/2}(1-p^{-2i+1}t^2) & 
{\rm if} \  n_1 \ {\rm even}, \\[2mm]
\displaystyle \prod_{i=1}^{(n-n_1-1)/2}(1-p^{-2i+1}t^2) &
{\rm if} \  n_1 \ {\rm odd}. 
\end{array}\right.
\]
\noindent
{\rm (2)} Let $p=2,$ and suppose that $B=2\Theta \bot B_1 \in {{\mathcal L}_{n-1,2}'}$  with $\Theta \in S_{n-n_1-2}({\bf Z}_2)_e \cap  GL_{n-n_1-2}({\bf Z}_2) $ and $B_1 \in S_{n_1+1}({\bf Z}_2)^{\times}.$  Then 

\begin{eqnarray*}
\lefteqn{B_p^{(1)}(B,t)} \\
&=&
\left\{
\begin{array}{cl}
\displaystyle (1-\overline{\xi}^{(1)}(B) p^{(n_1-n+1)/2}t)\prod_{i=1}^{(n-n_1-2)/2}(1-p^{-2i+1}t^2) & 
{\rm if} \  B_1 \ {\rm is \ of \ type \ (I)}, \\[2mm]
\displaystyle \prod_{i=1}^{(n-n_1-2)/2}(1-p^{-2i+1}t^2) & 
{\rm if} \  B_1 \ {\rm is \ of \ type \ (II) \ or \ (III)}. 
\end{array}\right. 
\end{eqnarray*}
\end{lems}
  
Let $m$ be a positive even integer and $j=0,1.$ For a non-degenerate half-integral matrix $T$ over ${\bf Z}_p$ of degree $m-j,$  put
$$R^{(j)}(T,X,t)=\sum_{W \in M_{m-j}({\bf Z}_p)^{\times}/GL_{m-j}({\bf Z}_p)}  \widetilde F_p^{(j)}(T[W],X)t^{\nu(\det W)}.$$
This type of formal power series  was first introduced by Andrianov \cite{A} to study the standard $L$-function of Siegel modular form of integral weight. Therefore we call it the formal power series of Andrianov type. (See also B{\"o}cherer \cite{Bo}.)
The following proposition follows from [\cite{K-K4}, Lemma 4.1.1 (1) ].  
\begin{props}
Let $m$ be a positive even integer and $j=0$ or $1.$  Let $T \in {\mathcal L}_{m-j,p}^{(j)}.$ Then  
 $$\sum_{B \in {\mathcal L}_{m-j,p}^{(j)}} {\widetilde F^{(j)}_p(B,X)\alpha_p(T,B) \over \alpha_p(B)}t^{\nu(\det B)}= t^{\nu(\det T)}R^{(j)}(T,X,p^{-m+j}t^2).$$
\end{props}
   
The following theorem is due to \cite{K-K3}.  

\begin{thms}
Let $T$ be an element of ${\mathcal L}_{n-1,p}^{(1)}.$ Then 
$$R^{(1)}(T,X,t)={B_p^{(1)}(T,p^{n/2-1}t)\widetilde G_p^{(1)}(T,X,t) \over  \prod_{j=1}^{n-1}(1-p^{j-1}X^{-1}t)(1-p^{j-1}Xt)}.$$
\end{thms}

In \cite{B-S}, B{\"o}cherer and Sato got a similar formula for $T \in {\mathcal L}_{n,p}.$  We note that the above formula for $p \not=2$ can be derived directly from 
Theorem 20.7 in \cite{Sh2}
(see also Zhuravlev \cite{Zh}). However, we note that we cannot use their  results  to prove the above formula for  $p=2.$

Then for $d_0 \in {\mathcal F}_p$ and   $\omega=\varepsilon^l$with $l=0,1,$  we define a formal power series $\widetilde R_{n-1}(d_0,\omega,X,Y,t)$ in $t$ by 
$$\widetilde R_{n-1}(d_0,\omega,X,Y,t)=\kappa(d_0,n-1,l)^{-1} t^{\delta_{2,p}(2-n)} \hspace*{-2.5mm}\sum_{B' \in {\mathcal L}_{n-1,p}^{(1)}(d_0)}\hspace*{-2.5mm} { \widetilde G_p^{(1)}(B',X,p^{-n}Yt^2)   \over \alpha_p(B') } $$
$$ \times Y^{-{\textfrak e}^{(1)}(B')} t^{\nu(\det B')} B_p^{(1)}(B',p^{-n/2-1}Yt^2) G_p^{(1)}(B',p^{-(n+1)/2}Y)\omega(B').$$
More precisely this is an element of ${\bf C}[X,X^{-1},Y^{1/2},Y^{-1/2}][[t]].$ 
 Now by Theorem 5,2,5, we can rewrite $H_{n-1}(\omega,d_0,X,Y,t)$ in terms of $\widetilde R_{n-1}(d_0,\omega,X,Y,t)$ in the following way: 

 \begin{thms}
 For $\omega=\varepsilon^l,$ we have 
 $$H_{n-1}(d_0,\omega,X,Y,t)={\widetilde R_{n-1}(d_0,\omega,X,Y,t) \over \prod_{j=1}^n(1-p^{j-1-n}XYt^2)(1-p^{j-1-n}X^{-1}Yt^2)}.$$
\end{thms}

 \begin{proof}  By [\cite{K-K4}, Lemma 4.2.2], we have
 \begin{eqnarray*}
\lefteqn{ \kappa(d_0,n-1,l)t^{\delta_{2,p}(n-2)}H_{n-1}(d_0,\omega,X,Y,t)=\sum_{B \in {\mathcal L}_{n-1,p}^{(1)}(d_0)}{\widetilde F_p^{(1)}(B,X) \over \alpha_p(B)} \omega(B)t^{\nu(\det B)} } \\
&& \hspace*{-3mm}\times \sum_{B' \in {\mathcal L}_{n-1,p}^{(1)}} {Y^{-{\textfrak e}^{(1)}(B')}G_p^{(1)}(B',p^{-(n+1)/2}Y) \alpha_p(B',B) \over \alpha_p(B')} (p^{-1}Y)^{(\nu(\det B)-\nu(\det B'))/2}.
\end{eqnarray*}
Let $B$ and $B'$ be elements of ${\mathcal L}_{n-1,p}^{(1)},$ and suppose that $\alpha_p(B',B) \not=0.$ Then we note that $B \in {\mathcal L}_{n-1,p}^{(1)}(d_0)$ if and only if $B' \in {\mathcal L}_{n-1,p}^{(1)}(d_0).$ Hence by Proposition 5.1.2 and Theorem 5.1.3 we have 
 \begin{eqnarray*}
\lefteqn{\kappa(d_0,n-1,l)t^{\delta_{2,p}(n-2)} H_{n-1}(d_0,\omega,X,Y,t)} \\
&=& \sum_{B' \in {\mathcal L}_{n-1,p}^{(1)}(d_0)} \hspace*{-2.5mm}{G_p^{(1)}(B',p^{-(n+1)/2}Y)Y^{-{\textfrak e}^{(1)}(B')} \over \alpha_p(B') } (pY^{-1})^{\nu(\det B')/2}\omega(B') \\
&& \times \sum_{B \in {\mathcal L}_{n-1,p}^{(1)}}  \hspace*{-2.5mm}{\widetilde F_p^{(1)}(B,X) \alpha_p(B',B) \over \alpha_p(B)}(t^2p^{-1}Y)^{\nu(\det B)/2} \\
&=&\sum_{B' \in {\mathcal L}_{n-1,p}^{(1)}(d_0)} { G_p^{(1)}(B',p^{-(n+1)/2}Y)Y^{-{\textfrak e}^{(1)}(B')} \over \alpha_p(B') } t^{\nu(\det B')}\omega(B')R^{(1)}(B',X,t^2Yp^{-n}) \\
&=&\sum_{B' \in {\mathcal L}_{n-1,p}^{(1)}(d_0)} { \widetilde G_p^{(1)}(B',X,p^{-n}Yt^2)   \over \alpha_p(B') } \omega(B') Y^{-{\textfrak e}^{(1)}(B')} t^{\nu(\det B')}\\
&& \times {B_p^{(1)}(B',p^{-n/2-1}Yt^2) G_p^{(1)}(B',p^{-(n+1)/2}Y) \over \prod_{j=1}^n(1-p^{j-1-n}XYt^2)(1-p^{j-1-n}X^{-1}Yt^2)}.
\end{eqnarray*} 
This proves the assertion. 
\end{proof}
The polynomials $G_p^{(1)}(T,X)$ and $B_p^{(1)}(T,t)$ are expressed explicitly, and in particular they are determined by $[{\textfrak d}_T]$ and the $p$-rank of $T$ (cf. Lemma 4.2.1 of \cite {K-K4} and Lemma 5.1.1). Thus we can rewrite the power series $\widetilde R_{n-1}(d_0,\omega,X,Y,t)$  in more concise form (cf. Corollary to Theorem 5.2.8.)

\subsection{Formal power series of Koecher-Maa{\ss}   type and of modified Koecher-Maa{\ss}   type} 
  
  \noindent
  { }
  
  \bigskip
 
 Let  $r$ be a positive even integer. For $d_0 \in {\mathcal F}_p,j=0,1$ and $l=0,1,$ we  define a  formal power series $P_{r-j}^{(j)}(d_0,\omega,\xi,X,t)$ in $t$ by 
$$P_{r-j}^{(j)}(d_0,\omega,\xi,X,t)=\kappa(d_0,r-j,l)^{-1} t^{-a_{r-j,p}}\hspace*{-2.5mm}\sum_{B \in {\mathcal L}_{r-j,p}^{(j)}(d_0)}  \hspace*{-2.5mm}{\widetilde F_p^{(j)}(B,\xi,X) \over \alpha_p(B)}\omega(B)t^{\nu(\det B)}$$
for $\omega=\varepsilon^l$ with $l=0,1,$ 
 where $a_{i,p}=\delta_{2,p}(i-1)$ or $1$ according as $i$ is odd or even.
In particular we put
$P_{r-j}^{(j)}(d_0,\omega,X,t)=P_{r-j}^{(j)}(d_0,\omega,1,X,t).$
 This type of formal power series appears in an explicit formula of the Koecher-Maa{\ss}   series associated with 
 the Siegel Eisenstein series and the Duke-Imamo{\=g}lu-Ikeda lift (cf. \cite{I-K2}, \cite{I-K3},\cite{K-K4}). Therefore we say that this formal power series is of Koecher-Maa{\ss}   type.  
 
 For a variable $Y$ we introduce the symbol $Y^{1/2}$ so that $(Y^{1/2})^2=Y,$ and for an integer $a$ write $Y^{a/2}=(Y^{1/2})^a.$ Under this convention, we can write $Y^{-{\textfrak e}^{(j)}(T)}t^{\nu(\det T)}$ as $Y^{a_{r-j,p}/2}Y^{\nu(d_0)/2}(Y^{-1/2}t)^{\nu(\det T)}$ if $T \in {\mathcal L}_{r-j,p}^{(j)}(d_0),$ 
and  we sometimes write a power series 
$$P(Y,t)=\sum_{T \in {\mathcal L}_{r-j,p}^{(j)}(d_0)} a(T,Y) Y^{-{\textfrak e}^{(j)}(T)}t^{\nu(\det T)}
 \in {\bf C}[Y,Y^{-1}][[t]]$$ as 
$$P(Y,t)=Y^{a_{r-j,p}/2}Y^{\nu(d_0)/2}\sum_{T \in {\mathcal L}_{r-j,p}^{(j)}(d_0)} a(T,Y) (Y^{-1/2}t)^{\nu(\det T)}.$$
For $T \in {\mathcal L}_{r-j,p}^{(j)}$ let $\widetilde G_p^{(j)}(T,\xi,X,t)$ be the  polynomial defined in the previous subsection. Moreover for $\xi=\pm 1,$ and $j=0,1,$ we define a formal power series $\widetilde P_{r-j}^{(j)}(n;d_0,\omega,\xi,X,Y,t)$  in $t$ by 
\begin{eqnarray*}
\lefteqn{
\widetilde P_{r-j}^{(j)}(n;d_0,\omega,\xi,X,Y,t)=\kappa(d_0,r-j,l)^{-1}(tY^{-1/2})^{-a_{r-j,p}}Y^{\nu(d_0)/2}
} \\
&\times&  \hspace*{-2.5mm}\sum_{B' \in {\mathcal L}_{r,p}^{(j)}(d_0)}\hspace*{-2.5mm} {\widetilde G_p^{(j)}(B',\xi,X,p^{-n}t^2Y) \over \alpha_p(B') } \omega(B')(tY^{-1/2})^{\nu(\det (B'))}
\end{eqnarray*}
for $\omega=\varepsilon^l.$ 
Here we make the convention that $\widetilde P_{0}^{(0)}(n;d_0,\omega,\xi,X,Y,t)=1$ or $0$ according as $\nu(d_0)=0$ or not.
We say that   the series $\widetilde P_{r-j}^{(j)}(n;d_0,\omega,\xi,X,Y,t)$ is  of modified Koecher-Maa{\ss}   type.  
The relation between $\widetilde P_{r-j}^{(j)}(n;d_0,\omega,\xi,X,Y,t)$ and $P_{r-j}^{(j)}(d_0,\omega,\xi,X,t)$ will be given in the following proposition:
\begin{props}
Let $r$ be a positive even integer. Let $\omega=\varepsilon^l$ with $l=0,1,$ and $j=0,1.$ Then 
$$\widetilde P_{r-j}^{(j)}(n;d_0,\omega,\xi,X,Y,t)= Y^{\nu(d_0)/2}P_{r-j}^{(j)}(d_0,\omega,\xi,X,tY^{-1/2}) \prod_{i=1}^{r-j} (1-t^4p^{-n-r+j-2+i}).$$
\end{props}

\begin{proof}  For $i=0,...,r-j$ put 
$$\widetilde P_{r-j,i}^{(j)}(d_0,\omega,\xi,X,t)=\kappa(d_0,r-j,l)^{-1}t^{-a_{r-j,p}}$$
$$ \times \sum_{B \in {\mathcal L}_{r-j,p}^{(j)}(d_0)} \sum_{D \in  GL_{r-j}({\bf Z}_p) \backslash {\mathcal D}_{r-j,i}} {\widetilde F_p^{(j)}(B[D^{-1}],\xi,X) \over \alpha_p(B)}\omega(B)t^{\nu(\det B)}.$$
Then we have 
\begin{eqnarray*}
\lefteqn{
\widetilde P_{r-j}^{(j)}(n;d_0,\omega,\xi,X,Y,t)
} \\
&=&\sum_{i=0}^{r-j} (-1)^i p^{i(i-1)/2}(p^{-n}t^2Y)^i  Y^{\nu(d_0)/2}\widetilde P_{r-j,i}^{(j)}(d_0,\omega,\xi,X,tY^{-1/2}) . 
\end{eqnarray*}
We have 
$$\widetilde P_{r-j,i}^{(j)}(d_0,\omega,\xi,X,t)= \sum_{B \in {\mathcal L}_{r-j,p}^{(j)}(d_0)} {\omega(B) \over \alpha_p(B)}t^{\nu(\det B)}$$
 $$ \times \sum_{B' \in {\mathcal L}_{r-j,p}^{(j)}} \widetilde F_p^{(j)}(B',\xi,X) \#(\widetilde \Omega(B',B,i)/GL_{r-j}({\bf Z}_p))  ,$$
where $\widetilde \Omega(B',B,i)=\{D \in {\mathcal D}_{r-j,i} \ | \ B'[D^{-1}] \sim B \}.$ 
Hence by [\cite{K-K4}, Lemma 4.1.1 (2)] we have
\begin{eqnarray*}
\lefteqn{\widetilde P_{r-j,i}^{(j)}(d_0,\omega,\xi,X,t)} \\
&=& \sum_{B \in {\mathcal L}_{r-j,p}^{(j)}(d_0)} {1 \over \alpha_p(B)} \sum_{B' \in {\mathcal L}_{r-j,p}^{(j)}} {\widetilde F_p^{(j)}(B',\xi,X) \alpha_p(B',B,i) \over  \alpha_p(B')}\omega(B) \\
&& \hspace*{5mm}\times \,p^{-(\nu(\det B)-\nu(\det B'))/2}t^{\nu(\det B)},
\end{eqnarray*}
where 
$$\alpha_p(B',B,i)=2^{-1}\lim_{e \to \infty}p^{-(r-j)(r-j-1)e/2} \#\{ \overline {X} \in {\mathcal A}_e(B',B) \ | \ X \in {\mathcal D}_{r-j,i} \}.$$
Let $B$ and $B'$ be elements of ${\mathcal L}_{r-j,p}^{(j)},$ and suppose that $\alpha_p(B',B,i) \not=0.$ Then we note that $B \in {\mathcal L}_{r-j,p}^{(j)}(d_0)$ if and only if $B' \in {\mathcal L}_{r-j,p}^{(j)}(d_0).$ Hence by [\cite{K-K4}, Lemma 4.1.1 (1)] we have
\begin{flushleft}$\widetilde P_{r-j,i}^{(j)}(d_0,\omega,\xi,X,t)$
\end{flushleft}
$$=\sum_{B' \in {\mathcal L}_{r-j,p}^{(j)}(d_0) } {\widetilde F_p^{(j)}(B',\xi,X)  \over  \alpha_p(B')} p^{\nu(\det B')/2} \omega(B')\sum_{B \in {\mathcal L}_{r-j,p}^{(j)}}  (tp^{-1/2})^{\nu(\det B)} {\alpha_p(B',B,i) \over \alpha_p(B)}$$
$$=\sum_{B' \in {\mathcal L}_{r-j,p}^{(j)}(d_0) } {\widetilde F_p^{(j)}(B',\xi,X)  \over  \alpha_p(B')} p^{\nu(\det B')/2}(tp^{-1/2})^{\nu(\det B')} (t^2p^{-r+j-1})^{i}\# (GL_{r-j}({\bf Z}_p) \backslash {\mathcal D}_{r-j,i}).$$
By Lemma 3.2.18 in 
\cite{A}, we have
$$\# (GL_{r-j}({\bf Z}_p) \backslash{\mathcal D}_{r-j,i})={\phi_{r-j}(p) \over \phi_i(p)\phi_{r-j-i}(p)}.$$
Hence 
\begin{eqnarray*}
\lefteqn{\widetilde P_{r-j,i}^{(j)}(d_0,\omega,\xi,X,t)} \\
&=& \sum_{B'  \in {\mathcal L}_{r-j,p}^{(j)}(d_0) } {\widetilde F_p^{(j)}(B',\xi,X)  \over  \alpha_p(B')} \omega(B')t^{\nu(\det B')}
{\phi_{r-j}(p) \over \phi_i(p)\phi_{r-j-i}(p)} (t^2p^{-r+j-1})^{i} \\
&=& {\phi_{r-j}(p) \over \phi_i(p)\phi_{r-j-i}(p)} P_{r-j}^{(j)}(d_0,\omega,\xi,X,t)(t^2p^{-r+j-1})^{i}. 
\end{eqnarray*}
 Thus, by (3.2.34) of \cite{A}, we have 
\begin{eqnarray*}
\lefteqn{\widetilde P_{r-j}^{(j)}(n;d_0,\omega,\xi,X,t) =Y^{\nu(d_0)/2}}\\
&\times &\sum_{i=0}^{r-j}(-1)^i p^{i(i+1)/2}(p^{-n-r+j-2}t^4)^i {\phi_{r-j}(p) \over \phi_i(p) \phi_{r-j-i}(p)} P_{r-j}^{(j)}(d_0,\omega,\xi,X,tY^{-1/2}) \\
&=& Y^{\nu(d_0)/2}P_{r-j}^{(j)}(d_0,\omega,\xi,X,tY^{-1/2}) \prod_{i=1}^{r-j} (1-t^4p^{-n-r+j-2+i}). 
\end{eqnarray*}
\end{proof}  
We give explicit formulas for $P_{r-j}^{(j)}(d_0,\varepsilon^l,\xi,X,t)$ for $j=0,1,l=0,1$ and $\xi=\pm 1.$
\begin{thms}
Let $d_0 \in {\mathcal F}_{p}$ and $\xi_0=\chi(d_0).$  

\noindent
{\rm (1)} Let $r$ be even. Then 
  $$P_r^{(0)}(d_0,\iota,X,t)={(p^{-1}t)^{\nu(d_0)} \over \phi_{r/2-1}(p^{-2})(1-p^{-r/2}\xi_0)}$$
  $$\times  {(1+t^2p^{-r/2-3/2})(1+t^2p^{-r/2-5/2}\xi_0^2)-
  \xi_0 t^2p^{-r/2-2}(X+X^{-1}+p^{1/2-r/2} +p^{-1/2+r/2}) \over 
  (1-p^{-2}Xt^2)(1-p^{-2}X^{-1}t^2)\prod_{i=1}^{r/2} (1-t^2p^{-2i-1}X)(1-t^2p^{-2i-1}X^{-1})  },$$
  and
  $$P_r^{(0)}(d_0,\varepsilon,X,t)={1  \over \phi_{r/2-1}(p^{-2})(1-p^{-r/2}\xi_0)}{\xi_0^2 \over 
  \prod_{i=1}^{r/2} (1-t^2p^{-2i}X)(1-t^2p^{-2i}X^{-1})  }.$$     
   \noindent
   {\rm (2)}  Let $r$ be even. Then  
  $$P_{r-1}^{(1)}(d_0, \iota,X,t)$$
  $$= {(p^{-1}t)^{\nu(d_0)}(1-\xi_0 t^2p^{-5/2} ) \over 
  (1-t^2p^{-2}X)(1-t^2p^{-2}X^{-1})\prod_{i=1}^{(r-2)/2} (1-t^2p^{-2i-1}X)(1-t^2p^{-2i-1}X^{-1}) \phi_{(r-2)/2}(p^{-2})},$$
  and
  $$P_{r-1}^{(1)}(d_0,\varepsilon,X,t)$$
  $$= {(p^{-1}t)^{\nu(d_0)}(1-\xi_0 t^2p^{(-1/2-r)}) \over 
  \prod_{i=1}^{r/2} (1-t^2p^{-2i}X)(1-t^2p^{-2i}X^{-1}) \phi_{(r-2)/2}(p^{-2})}.$$
\end{thms}  
\begin{proof} The assertions (1) and  (2) are due to  [\cite{Kat5}, Proposition 4.3], and to  [\cite{K-K4}, Theorem 4.4.1], respectively.
\end{proof}
\begin{xcor}
 Let $\xi=\pm 1.$ 

\noindent
{\rm (1)} Let $r$ be even. Then 
  $$P_r^{(0)}(d_0,\iota,\xi,X,t)={(p^{-1}t)^{\nu(d_0)} \over \phi_{r/2-1}(p^{-2})(1-p^{-r/2}\xi_0)}$$
  $$\times \{(1+t^2p^{-r/2-3/2}\xi)(1+t^2p^{-r/2-5/2}\xi\xi_0^2)$$
  $$-\xi_0 t^2p^{-r/2-2}(X+X^{-1}+p^{1/2-r/2}\xi +p^{-1/2+r/2}\xi) \}$$
$$ \times { 1 \over 
  (1-p^{-2}Xt^2)(1-p^{-2}X^{-1}t^2)\prod_{i=1}^{r/2} (1-t^2p^{-2i-1}X)(1-t^2p^{-2i-1}X^{-1})  },$$
  and
  $$P_r^{(0)}(d_0,\varepsilon,\xi,X,t)={1  \over \phi_{r/2-1}(p^{-2})(1-p^{-r/2}\xi_0)}{\xi_0^2 \over 
  \prod_{i=1}^{r/2} (1-t^2p^{-2i}X)(1-t^2p^{-2i}X^{-1})  }.$$     
   \noindent
   {\rm (2)}  Let $r$ be even. Then  
  $$P_{r-1}^{(1)}(d_0, \iota,\xi,X,t)$$
  $$= {(p^{-1}t)^{\nu(d_0)}(1-\xi_0 t^2p^{-5/2}\xi ) \over 
  (1-t^2p^{-2}X)(1-t^2p^{-2}X^{-1})\prod_{i=1}^{(r-2)/2} (1-t^2p^{-2i-1}X)(1-t^2p^{-2i-1}X^{-1}) \phi_{(r-2)/2}(p^{-2})},$$
  and
  $$P_{r-1}^{(1)}(d_0,\varepsilon,\xi,X,t)= {(p^{-1}t)^{\nu(d_0)}(1-\xi_0 t^2p^{(-1/2-r)}\xi ) \over 
  \prod_{i=1}^{r/2} (1-t^2p^{-2i}X)(1-t^2p^{-2i}X^{-1}) \phi_{(r-2)/2}(p^{-2})}.$$
\end{xcor}  

 \begin{proof} Put 
 $$S_{r-j}^{(j)}(d_0,\omega,\xi,X,t)=\sum_{T \in {\mathcal L}_{r-j,p}^{(j)}} {\widetilde F(T,\xi,X) \over \alpha_p(T)} t^{{\textfrak e}^{(j)}(T)},$$
 and
 $$S_{r-j}^{(j)}(d_0,\omega,X,t)=\sum_{T \in {\mathcal L}_{r-j,p}^{(j)}} {\widetilde F(T,X) \over \alpha_p(T)} t^{{\textfrak e}^{(j)}(T)}.$$
 Then we have 
$$P_{r-j}^{(j)}(d_0,\omega,\xi,X,t)=t^{\nu(d_0)}S_{r-j}^{(j)}(d_0,\omega,\xi,X,t^2) \ {\rm and} \ P_{r-j}^{(j)}(d_0,\omega,X,t)=t^{\nu(d_0)}S_{r-j}^{(j)}(d_0,\omega,X,t^2).$$
By definition we have 
$$S_{r-j}^{(j)}(d_0,\omega,\xi,X,t^2)=S_{r-j}^{(j)}(d_0,\omega,\xi X,\xi t^2).$$
Thus the assertion follows from the above theorem.
 \end{proof}


Now let $r$ be an even integer, and for $j=0,1,$  we consider partial series of $\widetilde P_{r-j}^{(j)}(n;d_0,\omega,\xi,X,Y,t):$  First let $p\not=2.$ Then put 
\begin{eqnarray*}
\lefteqn{Q_{r}^{(0)}(n;d_0,\varepsilon^l,\xi,X,Y,t)  = Y^{\nu(d_0)/2}} \\
 &\times & \sum_{B' \in S_{r}({\bf Z}_p,d_0) \cap S_{r}({\bf Z}_p)} {\widetilde G_p^{(0)}(pB',\xi,X,p^{-n}t^2Y) \over \alpha_p(pB') } \varepsilon(pB')^l(tY^{-1/2})^{\nu(\det pB')}, 
 \end{eqnarray*}
and
\begin{eqnarray*}
\lefteqn{Q_{r-1}^{(1)}(n;d_0,\varepsilon^l,\xi,X,Y,t)= \kappa(d_0,r-1,l)^{-1}  Y^{\nu(d_0)/2}}\\
 &\times& \sum_{B' \in p^{-1}S_{r-1}({\bf Z}_p,d_0) \cap S_{r-1}({\bf Z}_p)} {\widetilde G_p^{(1)}(pB',\xi,X,p^{-n}t^2Y) \over \alpha_p(pB') } \varepsilon(pB')^l(tY^{-1/2})^{\nu(\det pB')}. 
 \end{eqnarray*}
Next let $p=2.$  Then put
\begin{eqnarray*}
\lefteqn{Q_{r-1}^{(1)}(n;d_0,\varepsilon^l,\xi,X,Y,t)=\kappa(d_0,r-1,l)^{-1} (tY^{-1/2})^{\delta_{2,p}(2-n)} Y^{\nu(d_0)/2}} \\
& \times& \sum_{B' \in S_{r-1}({\bf Z}_2,d_0) \cap S_{r-1}({\bf Z}_2)} {\widetilde G_2^{(1)}(4B',\xi,X,2^{-n}t^2Y) \over \alpha_2(4B') } \varepsilon(4B')^l(tY^{-1/2})^{\nu_2(\det (4B'))},
\end{eqnarray*} 
and 
\begin{eqnarray*}
\lefteqn{Q_{r}^{(0)}(n;d_0,\varepsilon^l,\xi,X,Y,t)=\kappa(d_0,r,l)^{-1} Y^{\nu(d_0)/2}}\\ 
&\times& \sum_{B' \in  S_{r}({\bf Z}_2,d_0) \cap S_r({\bf Z}_2)_e} {\widetilde G_2^{(0)}(2B',\xi,X,2^{-n}t^2Y) \over \alpha_2(2B') } \varepsilon(B')^l(tY^{-1/2})^{\nu(\det (2B'))}. 
\end{eqnarray*}
Here we make the convention that $Q_{0}^{(0)}(n;d_0,\varepsilon^l,\xi,X,Y,t)=1$  or $0$ according as $\nu(d_0)=0$ or not. 

A non-degenerate square matrix $D=(d_{ij})_{m \times m}$ with entries in ${\bf Z}_p$ is said to be reduced if $D$ satisfies the following two conditions:\begin{enumerate}
\item[(a)] For 
$i=j$, $d_{ii}=p^{e_{i}}$ with a non-negative integer $e_i$; \vspace*{1mm}
 \item[(b)] For 
$i\ne j$, $d_{ij}$ is a non-negative integer satisfying 
$ d_{ij} \le p^{e_j}-1$ if $i <j$ and $d_{ij}=0$ if $i >j$. 
\end{enumerate}
It is well known that  we can take the set of all reduced matrices as a  complete set of representatives of $GL_m({\bf Z}_p) \backslash M_m({\bf Z}_p)^{\times}.$  

To consider the relation between 
\begin{center} 
$\widetilde P_{r-j}^{(j)}(n;d_0,\varepsilon^l,\xi,X,Y,t) \ $ and $\ Q_{r-j}^{(j)}(n;d_0,\varepsilon^l,\xi,X,Y,t),$ 
\end{center}
and to express $\widetilde R_{n-1}(d_0,\varepsilon^l,X,Y,t)$ in terms of 
$\widetilde P_{r-j}^{(j)}(n;d_0,\varepsilon^l,\xi,X,Y,t),$ we give some preliminary results.
\begin{lems}
 Let $p \not=2.$ Let $m$  be an even integer, and $r$ an integer such that $0 \le r \le m.$  Let $d \in {\mathcal U}$ and $\xi_0=\pm 1.$
  
  \noindent
  {\rm (1)}  Suppose that $r$ is even.
  
\begin{enumerate}
\item[{\rm (1.1)}] Let $B' \in S_{r}({\bf Z}_p)^{\times}.$ Then  
  $$\widetilde G_p^{(0)}(\Theta_{m-r,d} \bot  pB',\xi_0,X,t)=\widetilde G_p^{(0)}( pB', \xi_0 \chi(d),X,t).$$
  
\item[{\rm (1.2)}] Let $B' \in S_{r-1}({\bf Z}_p)^{\times}.$ Then 
  $$\widetilde G_p^{(1)}(\Theta_{m-r,d} \bot  pB',\xi_0,X,t)=\widetilde G_p^{(1)}(  pdB',\xi_0 , X,t).$$

\end{enumerate}  
  \noindent
  {\rm (2)}  Suppose that $r$ is odd. 
\begin{enumerate}
\item[{\rm (2.1)}] Let $B' \in S_{r}({\bf Z}_p)^{\times}.$ Then  
  $$\widetilde G_p^{(0)}(\Theta_{m-r,d} \bot pB',\xi_0,X,t)=\widetilde G_p^{(1)}( -pdB',\xi_0, X,t).$$

\item[{\rm (2.2)}] Let $B' \in S_{r-1}({\bf Z}_p)^{\times}.$ Then   
  $$\widetilde G_p^{(1)}(\Theta_{m-r,d} \bot pB',\xi_0,X,t)=\widetilde G_p^{(0)}(pB',\xi_0 \chi(d),X,t).$$
\end{enumerate}
\noindent
{\rm (3)} Suppose that $r$ is even. Then we have
$$\widetilde G_p^{(0)}(d'B,\xi_0,X,t)=\widetilde G_p^{(0)}(B,\xi_0,X,t)$$
for $d' \in {\bf Z}_p^*,$ and $B \in S_r({\bf Z}_p)^{\times}.$ 
\end{lems}
    \begin{proof}   Let $m-r$ be even.
By [\cite{Kat4}, Proposition 3.2] we have 
   $$\widetilde F_p^{(0)}(\Theta_{m-r,d} \bot  pB',\xi_0,X)=\widetilde F_p^{(0)}(pB',\xi_0\chi(d),X)$$
   for $B' \in S_{r}({\bf Z}_p)^{\times}.$ We note that
$$\widetilde G_p^{(0)}(\Theta_{m-r,d} \bot pB',\xi_0,X,t)=\sum_{i=0}^{m} (-1)^i p^{i(i-1)/2}t^i $$
$$ \times \sum_{D \in GL_{m}({\bf Z}_p) \backslash \widetilde \Omega^{(0)}(\Theta_{m-r,d} \bot pB',i)} \widetilde F^{(0)}_p((\Theta_{m-r,d} \bot pB')[D^{-1}],\xi_0,X), $$
 where for $j=0,1$ and  $B \in {\mathcal L}_{m-j,p}^{(j)}$ put 
$$\widetilde \Omega^{(j)}(B,i)=\{W \in {\mathcal D}_{m-j,i} \ |  \ B[W^{-1}] \in {\mathcal L}_{m-j,p}^{(j)} \}.$$
 Thus the assertion (1.1) follows from  [\cite{K-K4}, Lemma 4.1.2 (1.1)].
   Furthermore we have 
   $$\widetilde F_p^{(1)}(\Theta_{m-r,d} \bot  pB',\xi_0,X)=\widetilde F_p(1 \bot \Theta_{m-r,d} \bot  pB',\xi_0,X)$$
   $$=\widetilde F_p(d \bot \Theta_{m-r} \bot   pB',\xi_0,X)=\widetilde F_p(1  \bot \Theta_{m-r} \bot  p d B',\xi_0,X) $$
   $$=\widetilde F_p(1  \bot p d B',\xi_0,X)=\widetilde F_p^{(1)}(p d B',\xi_0,X) $$
   for $B' \in S_{r-1}({\bf Z}_p)^{\times}.$ Thus the assertion (1.2) follows from [\cite{K-K4}, Lemma 4.1.2 (1.2)]. The other assertions can be proved in a similar way.
   \end{proof}
   
\begin{lems}  
Let $p =2.$ Let $m$ and $r$ be even integers such that $0 \le r \le m,$  and $\xi_0=\pm 1.$   

\noindent 
{\rm (1)} Let $d \in {\mathcal U}.$
\begin{enumerate}
\item[{\rm (1.1)}] Let $B' \in S_r({\bf Z}_2)^{\times}.$ Then  
$$\widetilde G_2^{(0)}(\Theta_{m-r,d} \bot  2B',\xi_0,X,t)=\widetilde G_2^{(0)}( 2B',\xi_0 \chi(d), X,t),$$
  
\item[{\rm (1.2)}] Let $B' \in S_{r-1}({\bf Z}_2)^{\times}.$ Then  
  $$\widetilde G_2^{(1)}(2\Theta_{m-r,d} \bot 4B',\xi_0,X,t)=\widetilde G_2^{(1)}(4dB', \xi_0, X,t).$$
\end{enumerate}  
\noindent  
{\rm (2)} 
\begin{enumerate}
\item[{\rm (2.1)}]  Let $a \in {\mathcal U}$ and $B' \in S_{r}({\bf Z}_2)^{\times}.$ 
Then  
  $$\widetilde G_2^{(1)}(-a \bot 2\Theta_{m-r-2} \bot 4B',\xi_0,X,t)=\widetilde G_2^{(0)}(2B',\xi_0 \chi(a), X,t).$$
  
\item[{\rm (2.2)}] Let $B' \in S_{r-1}({\bf Z}_2)^{\times}$ and $a \in {\bf Z}_2^*.$ Then   
  $$\widetilde G_2^{(0)}(  \Theta_{m-r} \bot 2a \bot 2B',\xi_0,X,t)=\widetilde G_2^{(1)}(4aB',\xi_0, X,t).$$
\end{enumerate}
\noindent
{\rm (3)} We have
$$\widetilde G_2^{(0)}(d'B,\xi_0,X,t)=\widetilde G_2^{(0)}(B,\xi_0,X,t)$$
for $d' \in {\bf Z}_2^*,$ and $B \in S_r({\bf Z}_2)^{\times}.$ 

\noindent
{\rm (4)} Let $u_0 \in {\bf Z}_2^*$ and $B_1 \in S_{r-2}({\bf Z}_2)^{\times}.$ Then  
$$\widetilde G_2^{(1)}(u_0 \bot  5B_1,\xi_0,X,t)=\widetilde G_2^{(1)}( u_0 \bot B_1,\xi_0, X,t).$$
 \end{lems}
  
  \begin{proof} 
  All the assertions  except (4) can be proved in a way similar to Lemma 5.2.3. To prove (4), we first note that
  $GL_{m-1}({\bf Z}_2)  \backslash \widetilde \Omega(u_0 \bot 5B_1,i)=GL_{m-1}({\bf Z}_2)  \backslash \widetilde \Omega(u_0 \bot B_1,i)$ for 
  $i=0,\cdots,m-1.$ 
Hence it suffices to prove
$$\widetilde F_2^{(1)}((u_0 \bot 5B_1)[D^{-1}],\xi_0,X)= \widetilde F_2^{(1)}((u_0 \bot B_1)[D^{-1}],\xi_0,X)$$
for $D \in \widetilde \Omega(u_0 \bot B_1,i).$ 
We may assume $D$ is reduced. Since we have $u_0 \in {\bf Z}_2^*$ we have
$D=\mattwo(1;d;O;D_1)$ with $d \in M_{1,m-2}({\bf Z}_2)$ and $M_{m-2}({\bf Z}_2).$ We also note that $2D_1^{-1} \in M_{m-2}({\bf Z}_2).$
We have
$$\widetilde F_2^{(1)}((u_0 \bot 5B_1)[D^{-1}],\xi_0,X)=\widetilde F_2((1 \bot u_0 \bot 5B_1)[1 \bot D^{-1}],\xi_0,X)$$
$$=\widetilde F_2((1 \bot u_0 \bot 5B_1)\left[\smallmatthree(1;0;0;0;1;-dD_1^{-1};0;0;D_1^{-1}) \right],\xi_0,X)$$
$$=\widetilde F_2((5 \bot 5u_0 \bot B_1)\left[\smallmatthree(1;0;0;0;1;-dD_1^{-1};0;0;D_1^{-1}) \right],\xi_0,X).$$
We can easily see that there exits an element $U=(u_{ij}) \in GL_2({\bf Z}_2)$ such that $(1 \bot u_0)[U]=5 \bot 5u_0$ and $u_{12} \equiv 0, u_{22} \equiv 1 \ {\rm mod} \ 2.$ Then we have
$$\widetilde F_2^{(1)}((u_0 \bot 5B_1)[D^{-1}],\xi_0,X)=\widetilde F_2((1 \bot u_0 \bot B_1)[(1 \bot D^{-1})V],\xi_0,X),$$
where $V=\smallmatthree(u_{11};u_{12};-u_{12}dD_1^{-1};u_{21};u_{22};-u_{22}dD_1^{-1}+dD_1^{-1};0;0;1_{m-2}).$ By construction, we have $V \in GL_m({\bf Z}_2),$ and hence
we have 
$$\widetilde F_2((1 \bot u_0 \bot B_1)[(1 \bot D^{-1}])V],\xi_0,X)=\widetilde F_2^{(1)}((u_0 \bot B_1)[D^{-1}],\xi_0,X).$$
  \end{proof}

  \bigskip
  
 Let $\widetilde R_{n-1}(d_0,\omega,X,Y,t)$ be the formal power series defined at the beginning of Section 5. We express $\widetilde R_{n-1}(d_0,\omega,X,Y,t)$ in terms of  $Q_{2r}^{(0)}(n;d_0d, \omega, \chi(d), X,Y,t)$ and $Q_{2r+1}^{(1)}(n;d_0,\omega,1, X,Y,t).$ Henceforth, for $d_0 \in {\mathcal F}_p$ and non-negative integers $m,r$ such that $r \le m,$ put
 ${\mathcal U}(m,r,d_0)=\{1\},{\mathcal U} \cap \{d_0\},$ or ${\mathcal U}$ according as $r=0, \ r=m \ge 1,$ or $1 \le  r \le m-1.$ 
 Moreover, we sometimes abbreviate $S_r({\bf Z}_p)$ and $S_r({\bf Z}_p,d)$ as $S_{r,p}$ and $S_{r,p}(d),$ respectively.
Furthermore we abbreviate $S_r({\bf Z}_2)_x$ and $S_r({\bf Z}_2,d)_x$ as $S_{r,2;x}$ and $S_{r,2}(d)_x,$ respectively, for $x=e,o.$

\begin{thms}    
Let $d_0 \in {\mathcal F}_p$, and $\xi_0=\chi(d_0).$  For $d \in {\mathcal U}(n-1,n-2r-1,d_0)$ put
$$D_{2r}(d_0,d,Y,t)
={1-\xi_0 p^{-1/2}Y \over 1-p^{r-1/2}\chi(d) Y}(1-p^{-n-1/2+r}\chi(d)Yt^2) 
$$
\noindent
\begin{enumerate}
\item[{\rm (1)}] Let $\omega=\iota,$ or  $\nu(d_0)=0.$ Then  
$$\widetilde R_{n-1}(d_0,\omega,X,Y,t)$$
$$=\sum_{r=0}^{(n-2)/2} {\prod_{i=1}^{r} (1-p^{2i-1}Y^2) \prod_{i=1}^{(n-2r-2)/2}(1-p^{-2i-n-1}Y^2t^4) \over  2^{1-\delta_{0,r}} \phi_{(n-2r-2)/2}(p^{-2})}$$
$$ \times \sum_{d \in {\mathcal U}(n-1,n-2r-1,d_0)} D_{2r}(d_0,d,Y,t) Q_{2r}^{(0)}(n;d_0d, \omega, \chi(d), X,Y,t)$$
$$+ \sum_{r=0}^{(n-2)/2} {\prod_{i=1}^{r} (1-p^{2i-1}Y^2) \prod_{i=1}^{(n-2r-2)/2}(1-p^{-2i-n-1}Y^2t^4) \over \phi_{(n-2r-2)/2}(p^{-2})}$$
$$ \times  (1-\xi_0 p^{-1/2}Y) Q_{2r+1}^{(1)}(n;d_0,\omega,1, X,Y,t).$$

\item[{\rm (2)}] Let $\nu(d_0) > 0.$  Then  
$$\widetilde R_{n-1}(d_0,\varepsilon,X,Y,t)$$
$$= \sum_{r=0}^{(n-2)/2} {\prod_{i=1}^{r} (1-p^{2i-1}Y^2) \prod_{i=1}^{(n-2r-2)/2}(1-p^{-2i-n-1}Y^2t^4)  \over \phi_{(n-2r-2)/2}(p^{-2})}$$
$$ \times (1-\xi_0 p^{-1/2}Y) Q_{2r+1}^{(1)}(n;d_0,\varepsilon,1, X,Y,t).$$
\end{enumerate}

\end{thms}
  
\begin{proof} Let $p\not=2.$ Let $B$ be a symmetric matrix of degree $2r$ or $2r+1$ with entries in ${\bf Z}_p.$ Then we note that $\Theta_{n-2r-2,d} \bot pB$ belongs to ${\mathcal L}_{n-1,p}(d_0)$ if and only if $B \in S_{2r+1,p}(p^{-1}d_0d) \cap S_{2r+1,p}$, and that  $\Theta_{n-2r-1,d} \bot pB$ belongs to ${\mathcal L}_{n-1,p}(d_0)$ if and only if $B \in S_{2r,p}(d_0d) \cap S_{2r,p}.$ Thus by the theory of Jordan decompositions, for $\omega=\varepsilon^l$ we have
  $$\widetilde R_{n-1}(d_0,\omega,X,Y,t)=\kappa(d_0,n-1,l)^{-1}(tY^{-1/2})^{\delta_{2,p}(2-n)}$$
  $$\times \left\{ \sum_{r=0}^{(n-2)/2} \sum_{d \in {\mathcal U}(n-1,n-2r-2,d_0)} \sum_{B' \in p^{-1}S_{2r+1,p}(d_0d) \cap S_{2r+1,p}} 
   {G_p^{(1)}(\Theta_{n-2r-2,d} \bot pB',p^{-(n+1)/2}Y)   \over \alpha_p(\Theta_{n-2r-2,d} \bot pB') } \right.$$
   $$ \times B_p^{(1)}(\Theta_{n-2r-2,d} \bot pB',p^{-n/2-1}Yt^2) \widetilde G_p^{(1)}(\Theta_{n-2r-2,d} \bot pB',1,X,p^{-n}t^2Y) $$
  $$ \times \omega(\Theta_{n-2r-2,d} \bot pB') (tY^{-1/2})^{\nu(\det (pB'))} $$
  $$+\sum_{r=0}^{(n-2)/2} \sum_{d \in {\mathcal U}(n-1,n-2r-1,d_0)} \sum_{B' \in S_{2r,p}(d_0d) \cap S_{2r,p}}
  {G_p^{(1)}(\Theta_{n-2r-1,d} \bot pB',p^{-(n+1)/2}Y)   \over \alpha_p(\Theta_{n-2r-1,d} \bot pB') } $$
   $$ \times  B_p^{(1)}(\Theta_{n-2r-1,d} \bot pB',p^{-n/2-1}Yt^2) \widetilde G_p^{(1)}(\Theta_{n-2r-1,d} \bot pB',1,X,p^{-n}t^2Y) $$
  $$ \times \left. \omega(\Theta_{n-2r-1,d} \bot pB')(tY^{-1/2})^{\nu(\det (pB'))} \right\}.$$
  By [\cite{K-K4},  Lemma 4.2.1] and Lemma  5.1.1 we have
  $$G_p^{(1)}(\Theta_{n-2r-2,d} \bot pB',p^{-(n+1)/2}Y)B_p^{(1)}(\Theta_{n-2r-2,d} \bot pB',p^{-n/2-1}Yt^2)$$
  $$=\prod_{i=1}^{r} (1-p^{2i-1}Y^2) \prod_{i=1}^{(n-2r-2)/2}(1-p^{-2i-n-1}Y^2t^4) (1-\xi_0 p^{-1/2}Y) ,$$
  and 
  $$G_p^{(1)}(\Theta_{n-2r-1,d} \bot pB',p^{-(n+1)/2}Y)B_p^{(1)}(\Theta_{n-2r-1,d} \bot pB',p^{-n/2-1}Yt^2)$$
  $$=\prod_{i=1}^{r-1} (1-p^{2i-1}Y^2) \prod_{i=1}^{(n-2r-2)/2}(1-p^{-2i-n-1}Y^2t^4) D_{2r}(d_0,d,Y,t) .$$
    Put $H_{2i-1,\xi}^{(1)}(B)=\widetilde G_p^{(1)}(B,\xi,X,p^{-n}t^2Y)$ for $B \in S_{2i-1}({\bf Z}_p)^{\times},$ and $H_{2i,\xi}^{(0)}(B)=\widetilde G_p^{(0)}(B,\xi,X,p^{-n}t^2Y)$ for $B \in S_{2i}({\bf Z}_p)^{\times}$ and $\xi=\pm 1.$ Then $H_{2i-1,\xi}^{(1)}$ and $H_{2i,\xi}^{(0)}$ are $GL_{2i-1}({\bf Z}_p)$ -invariant functions on $S_{2i-1}({\bf Z}_p)^{\times}$ with values in ${\bf C}[X,X^{-1},Y,Y^{-1},t]$ and  satisfy the conditions (H-$p$-1) $\sim$ (H-$p$-5) in Section 4 of \cite{K-K4} by virture of Lemma 5.2.3. Thus the assertion (1) in case  $p\not=2$ follows from [\cite{K-K4}, Propositions 4.3.3 and 4.3.4]. 
    
  Next let $p=2.$ Let $B$ be a symmetric matrix of degree $2r$ or $2r+1$ with entries in ${\bf Z}_2,$ and $d \in {\mathcal U}.$ We note that $2\Theta_{n-2r-2,d} \bot 4B$ belongs to ${\mathcal L}_{n-1,2}(d_0)$ if and only if $B \in S_{2r+1,2}(d_0d) \cap S_{2r+1,2}$, and that  $-d \bot 2\Theta_{n-2r-2} \bot 4B$ belongs to ${\mathcal L}_{n-1,2}(d_0)$ if and only if $B \in S_{2r+2,2}(d_0d) \cap S_{2r+2,2}.$ Then, similarly to above, we have 
  $$\widetilde R_{n-1}(d_0,\omega,X,Y,t)=\kappa(d_0,n-1,l,tY^{-1/2})^{-1}$$
  $$\times \left\{\sum_{r=0}^{(n-2)/2} \left( \sum_{d \in {\mathcal U}(n-1,n-2r-2,d_0)} \sum_{B' \in S_{2r+1,2}(d_0d) \cap S_{2r+1,2;e}} 
  G_p^{(1)}(2\Theta_{n-2r-2,d} \bot 4B',2^{-(n+1)/2}Y) \right. \right.$$
  $$ \times  B_p^{(1)}(2\Theta_{n-2r-2,d} \bot 4B',p^{-n/2-1}Yt^2)  {\widetilde G_2^{(1)}(2\Theta_{n-2r-2,d} \bot 4B',1,X,2^{-n}t^2Y) \over \alpha_2(2\Theta_{n-2r-2,d} \bot 4B') } $$
  $$ \times \omega(2\Theta_{n-2r-2,d} \bot 4B')(tY^{-1/2})^{\nu(\det (4B'))+n-2r-2} $$ 
      $$+\sum_{B' \in S_{2r+1,2}(d_0) \cap S_{2r+1,2;o}} G_p^{(1)}(2\Theta_{n-2r-2} \bot 4B',2^{-(n+1)/2}Y) $$
  $$ \times B_p^{(1)}(2\Theta_{n-2r-2} \bot 4B',p^{-n/2-1}Yt^2) {\widetilde G_2^{(1)}(2\Theta_{n-2r-2} \bot 4B',1,X,2^{-n}t^2Y) \over \alpha_2(2\Theta_{n-2r-2} \bot 4B') } $$
  $$ \times \omega(2\Theta_{n-2r-2} \bot 4B')(tY^{-1/2})^{\nu(\det (4B'))+n-2r-2} $$
   $$+ \sum_{B' \in S_{2r+2,2}(d_0) \cap  S_{2r+2,2;o}} G_p^{(1)}(-1 \bot 2\Theta_{n-2r-4} \bot 4B',2^{-(n+1)/2}Y) $$  
    $$ \times  B_p^{(1)}(-1 \bot 2\Theta_{n-2r-4} \bot 4B',p^{-n/2-1}Yt^2) {\widetilde G_2^{(1)}(-1 \bot 2\Theta_{n-2r-4} \bot 4B',1,X,2^{-n}t^2Y) \over \alpha_2(-1 \bot 2\Theta_{n-2r-4} \bot 4B') } $$
   $$ \times \left. \omega(-1 \bot 2\Theta_{n-2r-4} \bot 4B')(tY^{-1/2})^{\nu(\det (4B'))+n-2r-4} \right)$$
$$+\sum_{r=0}^{(n-2)/2}  \sum_{d \in {\mathcal U}(n-1,n-2r-1,d_0)} \sum_{B' \in S_{2r,2}(d_0d) \cap S_{2r,2;e}} G_p^{(1)}(-d \bot 2\Theta_{n-2r-2} \bot 4B',2^{-(n+1)/2}Y) $$
$$ \times B_p^{(1)}(-d \bot 2\Theta_{n-2r-2} \bot 4B',p^{-n/2-1}Yt^2) {\widetilde G_2^{(1)}(-d \bot 2\Theta_{n-2r-2} \bot 4B',1,X,2^{-n}t^2Y) \over \alpha_2(-d \bot 2\Theta_{n-2r-2} \bot 4B') } $$
     $$ \times \left. \omega(-d \bot 2\Theta_{n-2r-2} \bot 4B')(tY^{-1/2})^{\nu(\det (4B'))+n-2r-2}\right\}.$$
     Thus the assertion (1) in case $p=2$ can be proved by using [\cite {K-K4}, Lemma 4.2.1], Lemmas 5.1.1 and 5.2.4, and [\cite{K-K4} , Propositions 4.3.3 and 4.3.4] in the same way as above. 
   Similarly the assertion (2) can be proved.
 \end{proof}
  
 \bigskip
 
  Now to rewrite the above theorem, first we express $\widetilde P_{m-1}^{(1)}(n;d_0,\omega, \eta, X,Y,t)$ in terms of  $Q_{2r+1}^{(1)}(n;d_0, \omega, \eta , X,Y,t)$ and $Q_{2r}^{(0)}(n;d_0 d,\omega,\eta, X,Y,t).$ 
  
\begin{props}
Let $m$ be an even integer. Let $d_0 \in {\mathcal F}_p,$ and $\eta=\pm 1.$ 
   
    \noindent
{\rm (1)}   {\rm (1.1)}  Let $l=0$ or $\nu(d_0)=0.$  Then 
  $$\widetilde P_{m-1}^{(1)}(n;d_0,\varepsilon^l, \eta, X,Y,t) = \sum_{r=0}^{(m-2)/2}  { Q_{2r+1}^{(1)}(n;d_0, \varepsilon^l, \eta , X,Y,t) \over  \phi_{(m-2-2r)/2}(p^{-2})}$$
  $$ +\sum_{r=0}^{(m-2)/2} \sum_{d \in {\mathcal U}(m-1,m-1-2r,d_0)}{Q_{2r}^{(0)}(n;d_0 d,\varepsilon^l,\eta \chi(d), X,Y,t)  \over 
 2^{1-\delta_{0,r}} \phi_{(m-2-2r)/2}(p^{-2})} .$$
   {\rm (1.2)}  Let $\nu(d_0) \ge  1.$ Then 
$$Q_{2r}^{(0)}(n;d_0 d,\varepsilon,\eta \chi(d), X,Y,t)=0$$
for any $d$ and 
  $$\widetilde P_{m-1}^{(1)}(n;d_0,\varepsilon, \eta, X,Y,t)= \sum_{r=0}^{(m-2)/2}  { Q_{2r+1}^{(1)}(n;d_0, \varepsilon, \eta , X,Y,t)
  \over    \phi_{(m-2-2r)/2}(p^{-2})}.$$
   {\rm (2)}  {\rm (2.1)}  Let $l=0$ or $\nu(d_0)=0.$ Then 
        $$\widetilde P_m^{(0)}(n;d_0,\varepsilon^l,\eta,X,Y,t)$$
  $$=\sum_{r=0}^{m/2}  \sum_{d \in {\mathcal U}(m,m-2r,d_0)}{ 1 +p^{(-m+2r)/2} \chi(d) \over   2^{1-\delta_{0,r}+\delta_{0,m} } \phi_{(m-2r)/2}(p^{-2})} Q_{2r}^{(0)}(n;d_0 d,\varepsilon^l, \eta \chi(d),X,Y,t)  $$
       $$ +  \sum_{r=0}^{(m-2)/2} { 1  \over    \phi_{(m-2r)/2}(p^{-2})} Q_{2r+1}^{(1)}(n;d_0,\varepsilon^l,\eta ,X,Y,t).$$

\noindent        
  {\rm (2.2)} Let $\nu(d_0) > 0.$  Then 
 $$\widetilde P_m^{(0)}(n;d_0,\varepsilon,\eta,X,Y,t)=0.$$
\end{props}  

\begin{proof}
 The assertion can be proved in a way similar to Theorem 5.2.5. 
 \end{proof}

\begin{xcor}
Let $r$ be a non-negative integer. Let $d_0$ be an element of ${\mathcal F}_{p}$ and $\xi=\pm 1.$ 
 
  \noindent
  {\rm (1)} Let $l=0$ or $\nu(d_0)=0.$ Then  
 $$Q_{2r}^{(0)}(n;d_0,\varepsilon^l,\xi, X,Y,t)$$
 $$=\sum_{m=0}^r \sum_{d \in {\mathcal U}(2r,2m,d_0)} {(-1)^m (\chi(d) +p^{-m})p^{-m^2} \over 2^{1-\delta_{0,r-m}+\delta_{0,r}}\phi_m(p^{-2}) }\widetilde P_{2r-2m}^{(0)}(n;d_0 d , \varepsilon^l,\xi \chi(d), X,Y,t)$$
$$+ \sum_{m=0}^{r-1}  {(-1)^{m+1} p^{-m-m^2} \over \phi_m(p^{-2}) }\widetilde P_{2r-2m-1}^{(1)}(n;d_0 , \varepsilon^l,\xi , X,Y,t)),$$
and  
 $$Q_{2r+1}^{(1)}(n;d_0,\varepsilon^l,\xi,X,Y,t)$$
 $$=\sum_{m=0}^r  {(-1)^m p^{-m-m^2} \over \phi_m(p^{-2}) }\widetilde P_{2r-2m+1}^{(1)}(n;d_0,\varepsilon^l,\xi,X,Y,t)$$
$$+ \sum_{m=0}^{r} \sum_{d \in {\mathcal U}(2r+1,2m+1,d_0)} {(-1)^{m+1} p^{-m-m^2} \over 2^{1-\delta_{0,r-m}}\phi_m(p^{-2}) }\widetilde P_{2r-2m}^{(0)}(n;d_0 d ,\varepsilon^l,\xi \chi(d),X,Y,t)).$$

\noindent
{\rm (2)} Let $\nu(d_0) >0.$ We have 
 $$Q_{2r+1}^{(1)}(n;d_0,\varepsilon,\xi,X,Y,t)=\sum_{m=0}^r  {(-1)^m p^{m-m^2} \over \phi_m(p^{-2}) }\widetilde P_{2r+1-2m}^{(1)}(n;d_0,\varepsilon,\xi,X,Y,t),$$ and
 $$Q_{2r}^{(0)}(n;d_0,\varepsilon,\xi,X,Y,t)=0.$$
\end{xcor}

\begin{proof}  We prove the assertion (1) by induction on $r.$ Clearly the assertion holds for $r=0.$ Let $r \ge 1$ and suppose that the assertion holds for any $r' <r.$ Fix $l$ and we simply write
$Q_{2i-j}^{(j)}(n;d,\varepsilon^l,\xi,X,Y,t)$ and $\widetilde P_{2i-j}^{(j)}(n;d,\varepsilon^l,\xi,X,Y,t)$ as $Q_{2i-j}^{(j)}(d;\xi)$ and 
 $\widetilde P_{2i-j}^{(j)}(d;\xi),$ respectively. 
 Then by Proposition 5.2.6 and the induction hypothesis we have
$$Q_{2r+1}^{(1)}(d_0; \xi)=\widetilde P_{2r+1}^{(1)}(d_0; \xi)$$
$$ - \sum_{i=1}^{r}  {1 \over  \phi_{i}(p^{-2})} \left\{\sum_{j=0}^{r-i} {(-1)^jp^{-j-j^2} \over \phi_j(p^{-2} )} \widetilde P_{2r-2i-2j+1}^{(1)}(d_0;\xi) \right.$$
 $$\left. +  \sum_{j=0}^{r-i} \sum_{d' \in {\mathcal U}(2r-2i+1,2j+1,d_0)} {(-1)^{j+1}p^{-j-j^2} \over 2^{1-\delta_{0,r-i-j}}\phi_j(p^{-2})}
 \widetilde P_{2r-2i-2j}^{(0)}(d_0d';\xi\chi(d'))\right\}$$
 $$ -\sum_{i=0}^{r-i} \sum_{d \in {\mathcal U}(2r+1,2i+1,d_0)}{ 1  \over 
 2^{1-\delta_{0,r-i}} \phi_{i}(p^{-2})} $$
$$ \times \left\{\sum_{j=0}^{r-i} \sum_{ d' \in {\mathcal U}(2r-2i,2j,d_0d)} {(-1)^j (\chi(d')+p^{-j})p^{-j^2} \over 2^{1-\delta_{0,r-i-j}} \phi_j(p^{-2})} \widetilde P_{2r-2i-2j}^{(0)}(d_0 dd';\xi \chi(d)\chi(d'))\right.$$
$$\left. +\sum_{j=0}^{r-i-1}  {(-1)^{j+1}p^{-j-j^2} \over  \phi_j(p^{-2})} \widetilde P_{2r-2i-2j-1}^{(1)}(d_0 d;\xi \chi(d))\right\}$$
By Proposition 5.2.1 and Corollary to Theorem 5.2.2 we have
$$\widetilde P_{2r-2i-2j-1}^{(1)}(d_0 d;\xi \chi(d))=\widetilde P_{2r-2i-2j-1}^{(1)}(d_0;\xi ) $$
for $d \in {\mathcal U}(2r+1,2i+1,d_0)$ and hence
$$\sum_{d \in {\mathcal U}(2r+1,2i+1,d_0)}\widetilde P_{2r-2i-2j-1}^{(1)}(d_0 d;\xi \chi(d))=0.$$
Moreover we have 
$$ \sum_{d \in {\mathcal U}(2r+1,2i+1,d_0)}\sum_{ d' \in {\mathcal U}(2r-2i,2j,d_0d)}
{(-1)^j (\chi(d')+p^{-j})p^{-j^2} \over  2^{1-\delta_{0,r-i}} 2^{1-\delta_{0,r-i-j}} } \widetilde P_{2r-2i-2j}^{(0)}(d_0 dd';\xi \chi(d)\chi(d'))$$
$$=\sum_{ d'' \in {\mathcal U}(2r,2i+2j,d_0)}
{(-1)^j p^{-j-j^2} \over  2^{1-\delta_{0,r-i-j}} } \widetilde P_{2r-2i-2j}^{(0)}(d_0 d'';\xi \chi(d'')).$$
Hence we have
$$Q_{2r+1}^{(1)}(d_0;\xi)=\widetilde P_{2r+1}^{(1)}(d_0;\xi)+\sum_{m=1}^r \widetilde P_{2r-2m+1}(d_0;\xi)A_m$$
$$-\sum_{m=1}^r\sum_{d \in {\mathcal U}(2r+1,2m+1,d_0)} {1 \over 2^{1-\delta_{0,r-m}}} \widetilde P_{2r-2m}^{(0)}(d_0d;\xi\chi(d))A_m$$
$$-\sum_{m=0}^r\sum_{d \in {\mathcal U}(2r,2m,d_0)}  {1 \over 2^{1-\delta_{0,r-m}}} \widetilde P_{2r-2m}^{(0)}(d_0d;\xi\chi(d))B_m,$$
where $\displaystyle A_m=-\sum_{j=0}^{m-1} {(-1)^jp^{-j-j^2} \over \phi_{m-j}(p^{-2})\phi_j(p^{-2})} ,$ and 
$\displaystyle B_m=-\sum_{j=0}^{m} {(-1)^jp^{-j-j^2} \over \phi_{m-j}(p^{-2})\phi_j(p^{-2})} .$
We have $\displaystyle A_m={(-1)^mp^{-m-m^2} \over \phi_m(p^{-2})}$ for $m \ge 1,$ and $B_m=1$ or $0$ according as
$m=0$ or $m \ge 1.$ Thus  we get the desired result for $Q_{2r+1}^{(1)}(n;d_0,\varepsilon^l, 
 \xi, X,Y,t).$ We also get the result for $Q_{2r}^{(0)}(n;d_0,\varepsilon^l, 
 \xi, X,Y,t),$ and this completes the induction. Similarly the assertion (2) can be proved.

\end{proof}

 The following lemma follows from [\cite{I-K3}, Lemma 3.4]:
\begin{lems}
Let $l$ be a positive integer, and $q,U$ and $Q$ variables. Then 
\begin{eqnarray*}
\lefteqn{
\prod_{i=1}^l (1-U^{-1}Qq^{-i+1})U^l } \\
&=&\sum_{m=0}^l {\phi_l(q^{-1}) \over \phi_{l-m}(q^{-1}) \phi_m(q^{-1})}\prod_{i=1}^{l-m} (1-Qq^{-i+1})\prod_{i=1}^m (1-Uq^{i-1}) (-1)^mq^{(m-m^2)/2}.
\end{eqnarray*}
\end{lems}

The following corollary follows directly from the above lemma, and will be used  in the proof of Theorem 5.2.8.

\begin{xcor}
 Let $t$ and $Y$ be variables, and $p$ a prime number.


{\rm (1)} For a non-negative integers $l \le (n-2)/2$ and $i_0$ we have
$$ \sum_{m=0}^{(n-2-2l)/2}(-1)^m p^{m-m^2}{\prod_{i=i_0}^{l+m-1} (1-p^{2i-1}Y^2) \prod_{i=1}^{(n-2l-2m-2)/2}(1-p^{-2i-n-1}Y^2t^4) \over \phi_m(p^{-2})\phi_{(n-2-2l)/2-m}(p^{-2})}$$
$$={\prod_{i=i_0}^{l-1}(1-p^{2i-1}Y^2)\prod_{i=1}^{(n-2l-2)/2} (1-p^{-2l-n-2i}t^4)(p^{2l-1}Y^2)^{(n-2l-2)/2} \over \phi_{(n-2-2l)/2}(p^{-2})}. $$

{\rm (2)} For non-negative integer $l \le (n-2)/2$ we have
$$ \sum_{m=0}^{(n-2-2l)/2}(-1)^m p^{m-m^2}{\prod_{i=1}^{l+m} (1-p^{2i-1}Y^2) \prod_{i=1}^{(n-2l-2m-2)/2}(1-p^{-2i-n-1}Y^2t^4) \over \phi_m(p^{-2})\phi_{(n-2-2l)/2-m}(p^{-2})}$$
$$={\prod_{i=1}^{l}(1-p^{2i-1}Y^2)\prod_{i=1}^{(n-2l-2)/2} (1-p^{-2l-n-2i-2}t^4)(p^{2l+1}Y^2)^{(n-2l-2)/2} \over \phi_{(n-2-2l)/2}(p^{-2})}.  $$

{\rm (3)} For a non-negative  integer $l \le (n-4)/2$ we have
$$ \sum_{m=0}^{(n-4-2l)/2}(-1)^m p^{m-m^2}{\prod_{i=1}^{l+m} (1-p^{2i-1}Y^2) \prod_{i=1}^{(n-2l-2m-4)/2}(1-p^{-2i-n-1}Y^2t^4) \over \phi_m(p^{-2})\phi_{(n-4-2l)/2-m}(p^{-2})}$$
$$={\prod_{i=1}^{l}(1-p^{2i-1}Y^2)\prod_{i=1}^{(n-2l-4)/2} (1-p^{-2l-n-2i-2}t^4)(p^{2l+1}Y^2)^{(n-2l-4)/2} \over \phi_{(n-4-2l)/2}(p^{-2})}  .$$
Throughout {\rm  (1)} $\sim$ {\rm (3)}, we understand that the product $\prod_{i=a}^b( * )=1$ if $a>b.$

\end{xcor}

\begin{thms}  
  Let the notation be as in Theorem 5.2.5. \\
{\rm (1)} Suppose  that $\nu(d_0)=0.$  
Put $\xi_0=\chi(d_0).$ Then   
$$\widetilde R_{n-1}(d_0,\omega,X,Y,t)=(1-p^{-n}t^2) $$
$$ \times \{\sum_{l=0}^{(n-2)/2}\sum_{d \in {\mathcal U}(n-1,n-1-2l,d_0)}\widetilde P_{2l}^{(0)}(n;d_0d,\omega,\chi(d),X,Y,t)  \prod_{i=1}^{(n-2-2l)/2}(1-p^{-2l-n-2i}t^4) $$
$$ \times { (p^{2l-1}Y^2)^{(n-2l-2)/2} \prod_{i=0}^{l-1} (1-p^{2i-1}Y^2)  p^{l-1/2}\chi(d)Y(1+\chi(d)Yp^{l-1/2}) \over 2^{1-\delta_{0,l}}(1+\xi_0p^{-1/2}Y)\phi_{(n-2l-2)/2}(p^{-2})} $$
$$+\sum_{l=0}^{(n-4)/2}  \widetilde P_{2l+1}^{(1)}(n;d_0,\omega,1,X,Y,t) \prod_{i=2}^{(n-4-2l)/2}(1-p^{-2l-n-2i}t^4)$$
$$ \times   { (p^{2l-1}Y^2)^{(n-2l-2)/2} \prod_{i=1}^{l} (1-p^{2i-1}Y^2) (1-\xi_0 p^{-1/2}Y) (1+p^{-2l-2}t^2) \over \phi_{(n-2l-2)/2}(p^{-2})} \}.$$
\noindent
{\rm (2)} Suppose  that $\nu(d_0) >0$ and  $\omega=\iota.$ 
Put $\xi_0=\chi(d_0).$ Then   
$$\widetilde R_{n-1}(d_0,\omega,X,Y,t)=(1-p^{-n}t^2) $$
$$ \times \{\sum_{l=1}^{(n-2)/2}\sum_{d \in {\mathcal U}(n-1,n-1-2l,d_0)}\widetilde P_{2l}^{(0)}(n;d_0d,\omega,\chi(d),X,Y,t)  \prod_{i=1}^{(n-2-2l)/2}(1-p^{-2l-n-2i}t^4) $$
$$ \times { (p^{2l-1}Y^2)^{(n-2l-2)/2} \prod_{i=1}^{l-1} (1-p^{2i-1}Y^2)  p^{l-1/2}\chi(d)Y(1+\chi(d)Yp^{l-1/2}) \over 2 \phi_{(n-2l-2)/2}(p^{-2})} $$
$$+\sum_{l=0}^{(n-4)/2}  \widetilde P_{2l+1}^{(1)}(n;d_0,\omega,1,X,Y,t) \prod_{i=2}^{(n-4-2l)/2}(1-p^{-2l-n-2i}t^4)$$
$$ \times   { (p^{2l-1}Y^2)^{(n-2l-2)/2} \prod_{i=1}^{l} (1-p^{2i-1}Y^2)  (1+p^{-2l-2}t^2) \over \phi_{(n-2l-2)/2}(p^{-2})} \}.$$
\noindent
{\rm (3)} Suppose that $\nu(d_0) >0$ and $\omega=\varepsilon.$ Then   
$$\widetilde R_{n-1}(d_0,\omega,X,Y,t)=\sum_{l=0}^{(n-2)/2}  \widetilde P_{2l+1}^{(1)}(n;d_0,\omega,1,X,Y,t)   $$
$$ \times {(p^{2l+1}Y^2)^{(n-2l-2)/2} \prod_{i=1}^{l} (1-p^{2i-1}Y^2)\prod_{i=1}^{(n-2-2l)/2}(1-p^{-2l-n-2i-2}t^4) \over \phi_{(n-2-2l)/2}(p^{-2})}.$$
\end{thms}

\begin{proof} Suppose that $\nu(d_0)=0$ or $\omega=\iota.$ Then by (1) of Theorem 5.2.5 and (1) of Corollary to Proposition 5.2.6, we have $$\widetilde R_{n-1}(d_0,\omega;X,Y,t)$$
 $$=\sum_{r=0}^{(n-2)/2} {\prod_{i=1}^{r} (1-p^{2i-1}Y^2) \prod_{i=1}^{(n-2r-2)/2}(1-p^{-2i-n-1}Y^2t^4) \over 2^{1-\delta_{0,r}}\phi_{(n-2r-2)/2}(p^{-2})}$$
$$ \times \sum_{d_1 \in {\mathcal U}(n-1,n-2r-1,d_0)} D_{2r}(d_0,d_1,Y,t)\{\sum_{m=0}^r \sum_{d_2 \in {\mathcal U}(2r,2m,d_0d_1)} {(-1)^m (\chi(d_2) +p^{-m})p^{-m^2} \over 2^{1-\delta_{0,r-m}+\delta_{0,r}}\phi_m(p^{-2}) }$$
$$ \times \widetilde P_{2r-2m}^{(0)}(n;d_0 d_1 d_2 , \omega, \chi(d_1)\chi(d_2), X,Y;t)$$
$$+ \sum_{m=0}^{r-1} {(-1)^{m+1} p^{-m-m^2} \over \phi_m(p^{-2}) }\widetilde P_{2r-2m-1}^{(1)}(n;d_0d_1,\omega, \chi(d_1), X,Y;t))\}$$
$$ +  \sum_{r=0}^{(n-2)/2} {\prod_{i=1}^{r} (1-p^{2i-1}Y^2) \prod_{i=1}^{(n-2r-2)/2}(1-p^{-2i-n-1}Y^2t^4) \over 2^{1-\delta_{0,r}}\phi_{(n-2r-2)/2}(p^{-2})} $$  
$$\times (1-\xi_0p^{-1/2}Y) \{\sum_{m=0}^r  {(-1)^m p^{-m}p^{-m^2} \over \phi_m(p^{-2}) }\widetilde P_{2r+1-2m}^{(1)}(n;d_0 ,\omega,1, X,Y,t)$$
$$+ \sum_{m=0}^{r} \sum_{d_2 \in {\mathcal U}(2r+1,2m+1,d_0)} {(-1)^{m+1} p^{-m-m^2} \over 2^{1-\delta_{0,r-m}}\phi_m(p^{-2}) }\widetilde P_{2r-2m}^{(0)}(n;d_0  d_2, \omega, \chi(d_2),X,Y,t)\}.$$
By Proposition 5.2.1 and Corollary to Theorem 5.2.2,  for any $d_1 \in {\mathcal U}$ we have 
$$\widetilde P_{2r+1-2m}^{(1)}(n;d_0 d_1, \omega,\chi(d_1),X,Y,t)=\widetilde P_{2r+1-2m}^{(1)}(n;d_0, \omega,1,X,Y,t).$$
Moreover, if $r >m \ge 0,$ then ${\mathcal U}(n-1,n-2r-1,d_0)={\mathcal U}(2r+1,2m+1,d_0d_1)= {\mathcal U}.$  
Hence  \\
{\rm (A)} \qquad $\widetilde R_{n-1}(d_0,\omega,X,Y,t)$\\
$$=\sum_{m=0}^{(n-2)/2} S(n;m,d_0,Y)  {\prod_{i=1}^{m}(1-p^{2i-1}Y^2)\prod_{i=1}^{(n-2-2m)/2}(1-p^{-2i-n-1}Y^2t^4)p^{-m^2}(-1)^m   \over \phi_m(p^{-2})\phi_{(n-2)/2-m}(p^{-2})} $$
$$+\sum_{l=1}^{(n-2)/2}  
\sum_{d \in {\mathcal U}} \widetilde P_{2l}^{(0)}(n;d_0 d,\omega,\chi(d),X,Y,t)$$
$$ \times \sum_{m=0}^{(n-2-2l)/2}\{{1 \over 2}\sum_{d_1 \in {\mathcal U}(2l+2m,2m,d_0)} D_{2l+2m}(d_0, d_1,Y,t)(\chi(d_1)\chi(d) +p^{-m})(-1)^mp^{-m^2}$$
$$-  (1-\xi_0 p^{-1/2}Y) (-1)^mp^{-m-m^2}\}$$
$$ \times {\prod_{i=1}^{l+m} (1-p^{2i-1}Y^2) \prod_{i=1}^{(n-2l-2m-2)/2}(1-p^{-2i-n-1}Y^2t^4) \over 2\phi_m(p^{-2})\phi_{(n-2-2l)/2-m}(p^{-2})}$$
$$+\sum_{l=0}^{(n-2)/2}  \widetilde P_{2l+1}^{(1)}(n;d_0, \omega, 1, X,Y,t)   $$
$$ \times \{\sum_{m=0}^{(n-2-2l)/2}  ( (1-\xi_0 p^{-1/2}Y)(-1)^mp^{-m-m^2})$$
$$ \times {\prod_{i=1}^{l+m} (1-p^{2i-1}Y^2) \prod_{i=1}^{(n-2l-2m-2)/2}(1-p^{-2i-n-1}Y^2t^4) \over \phi_m(p^{-2})\phi_{(n-2-2l)/2-m}(p^{-2})}$$
$$-\sum_{m=0}^{(n-4-2l)/2}  {1 \over 2} \sum_{d \in {\mathcal U}}D_{2l+2m+2}(d_0,d,Y,t)(-1)^m p^{-m-m^2}$$
$$ \times {\prod_{i=1}^{l+m} (1-p^{2i-1}Y^2) \prod_{i=1}^{(n-2l-2m-4)/2}(1-p^{-2i-n-1}Y^2t^4) \over \phi_m(p^{-2})\phi_{(n-4-2l)/2-m}(p^{-2})}\},$$
where 
$$S(n;m,d_0,Y)=\sum_{d_1 \in {\mathcal U}(n-1,n-2m-1,d_0)} {(\chi(d_0d_1)+p^{-m}) D_{2m}(d_0,d_1,Y,t) \over 2}
-(1-\xi_0 p^{-1/2}Y)p^{-m}  \ {\rm or} \ 0$$
according as $\nu(d_0)=0$ or not. We  have 
$$(1-\xi_0 p^{-1/2}Y)p^{-m}(1-p^{-2n+2l+2m+1}Y^2t^4)$$
$$-{1 \over 2}\sum_{d_1 \in {\mathcal U}} D_{2l+2m+2}(d_0,d_1,Y,t)p^{-m}(1-p^{-2n+2m+2l+2})(1-p^{2l+2m+1}Y^2)$$
$$=(1-p^{-1/2}\xi_0 Y)p^{-n+m+2l+2}(1-p^{-2n+2m+2l+1}Y^2t^4)$$
$$+(1-p^{-1/2}\xi_0 Y)(1-p^{-n+2m+2l+2})p^{2l+m-n+1}Y^2t^2(1-p^{-n}t^2)$$
for any $ 0 \le l \le (n-2)/2$ and $0 \le m \le (n-2l-2)/2.$
Furthermore  we have  
$${1 \over 2}\sum_{d_1 \in {\mathcal U}} D_{2l+2m}(d_0,d_1,Y,t)(\chi(d_1)\chi(d) +p^{-m}) -(1-\xi_0 p^{-1/2}Y) p^{-m}$$
$$={\chi(d)(1-p^{-1/2}\xi_0 Y)(1-p^{-n}t^2)p^{l+m-1/2}Y(1+\chi(d) Yp^{l-1/2}) \over 1-p^{2l+2m-1}Y^2}$$
for any $1 \le l \le (n-2)/2, 0 \le m \le (n-2l-2)/2$ and $d \in {\mathcal U}.$
Suppose that  $\nu(d_0)=0.$ Then for any  $m$ we have
$$  {1 \over 2}\sum_{d_1 \in {\mathcal U}(n-1,n-2m-1,d_0)}  D_{2m}(d_0,d_1,Y,t)(\chi(d_1)\chi(d_0)+p^{-m})-(1-\xi_0 p^{-1/2}Y) p^{-m}$$
$$= {\xi_0p^{m-1/2}Y(1-p^{-1}Y^2)(1-p^{-n}t^2) \over 1-p^{2m-1}Y^2}.$$
Remark that ${\mathcal U}(n-1,n-2l-1,d_0)={\mathcal U}$ for $l >0$ and ${\mathcal U}(n-1,n-1,d_0)=\{ d_0 \}.$ Hence  
$$\widetilde R_{n-1}(d_0,\omega,X,Y,t)=\xi_0 p^{-1/2}Y(1-p^{-n}t^2)$$
$$ \times \sum_{l=0}^{(n-2)/2} \sum_{d \in {\mathcal U}(n-1,n-2l-1,d_0)}\widetilde P_{2l}^{(0)}(n;d_0 d,\omega,\chi(d),X,Y,t)$$
$$ \times(1-p^{-n}t^2)p^{l-1/2}\chi(d) Y {1+\chi(d) Y p^{l-1/2} \over  1+\xi_0p^{-1/2}Y}$$
$$ \times \sum_{m=0}^{(n-2-2l)/2}(-1)^m p^{m-m^2}{\prod_{i=0}^{l+m-1} (1-p^{2i-1}Y^2) \prod_{i=1}^{(n-2l-2m-2)/2}(1-p^{-2i-n-1}Y^2t^4) \over 2^{1-\delta_{0,l}}\phi_m(p^{-2})\phi_{(n-2-2l)/2-m}(p^{-2})}$$
$$+\sum_{l=0}^{(n-2)/2} \widetilde P_{2l+1}^{(1)}(n;d_0,\omega,1,X,Y,t)$$
$$ \times \{ (1-\xi_0p^{-1/2}Y)p^{-n+2+2l}$$
$$ \times \sum_{m=0}^{(n-2-2l)/2}(-1)^m p^{m-m^2}{\prod_{i=1}^{l+m} (1-p^{2i-1}Y^2) \prod_{i=1}^{(n-2l-2m-2)/2}(1-p^{-2i-n-1}Y^2t^4) \over \phi_m(p^{-2})\phi_{(n-2-2l)/2-m}(p^{-2})}$$
$$  +Y^2t^2p^{2l-n+1}(1-\xi_0p^{-1/2}Y)(1-p^{-n}t^2)$$
$$ \times \sum_{m=0}^{(n-4-2l)/2}(-1)^m p^{m-m^2}{\prod_{i=1}^{l+m} (1-p^{2i-1}Y^2) \prod_{i=1}^{(n-2l-2m-4)/2}(1-p^{-2i-n-1}Y^2t^4) \over \phi_m(p^{-2})\phi_{(n-4-2l)/2-m}(p^{-2})} \}.$$
Thus  the assertion (1) follows from Corollary to Lemma 5.2.7. 
 Similarly the assertion (2) can be proved. 

Suppose that $\nu(d_0)>0$ and $\omega=\varepsilon.$  Then by  (2) of Theorem 5.2.5 and (2) of Corollary to Proposition 5.2.6, we have 
 $$\widetilde R_{n-1}(d_0,\omega,X,Y,t)=\sum_{l=0}^{(n-2)/2}  \widetilde P_{2l+1}^{(1)}(n;d_0,\omega,1,X,Y,t)   $$
$$ \times \sum_{m=0}^{(n-2-2l)/2}  ((-1)^m p^{m-m^2})$$
$$ \times {\prod_{i=1}^{l+m} (1-p^{2i-1}Y^2) \prod_{i=1}^{(n-2l-2m-2)/2}(1-p^{-2i-n-1}Y^2t^4) \over \phi_m(p^{-2})\phi_{(n-2-2l)/2-m}(p^{-2})}.$$ 
Thus the  assertion (3) follows from Corollary to Lemma 5.2.7.

\end{proof}

By Proposition 5.2.1 we immediately obtain:

\bigskip

\begin{xcor}
  Let the notation be as in Theorem 5.2.8. \\
{\rm (1)} Suppose  that $\nu(d_0)=0.$ 
 Then   
$$\widetilde R_{n-1}(d_0,\omega,X,Y,t)=(1-p^{-n}t^2)\prod_{i=1}^{n/2 -1}(1-p^{-2n+2i}t^4) $$
$$ \times \left\{ \sum_{l=0}^{(n-2)/2}\prod_{i=1}^{l}(1-p^{-n-2l-3+2i}t^4)  \sum_{d \in {\mathcal U}(n-1,n-1-2l,d_0)} P_{2l}^{(0)}(d_0d,\omega,\chi(d),X,tY^{-1/2})  \right.$$
$$ \times { (p^{2l-1}Y^2)^{(n-2l-2)/2} \prod_{i=0}^{l-1} (1-p^{2i-1}Y^2)  p^{l-1/2}\chi(d)Y(1+\chi(d)Yp^{l-1/2}) \over 2^{1-\delta_{0,l}}(1+\xi_0p^{-1/2}Y)\phi_{(n-2l-2)/2}(p^{-2})}$$
$$+\sum_{l=0}^{(n-2)/2}  \prod_{i=1}^{l}(1-p^{-n-2l-3+2i}t^4) P_{2l+1}^{(1)}(d_0,\omega,1,X,tY^{-1/2}) $$
$$ \times  \left. { (p^{2l-1}Y^2)^{(n-2l-2)/2} \prod_{i=1}^{l} (1-p^{2i-1}Y^2) (1-\xi_0 p^{-1/2}Y) (1+p^{-2l-2}t^2) \over \phi_{(n-2l-2)/2}(p^{-2})} \right\}.$$
\noindent
{\rm (2)} Suppose  that $\nu(d_0) >0$ and  $\omega=\iota.$ 
Put $\xi_0=\chi(d_0).$  Then   
$$\widetilde R_{n-1}(d_0,\omega,X,Y,t)=(1-p^{-n}t^2)\prod_{i=1}^{n/2 -1}(1-p^{-2n+2i}t^4) $$
$$ \times \left\{ \sum_{l=1}^{(n-2)/2}\prod_{i=1}^{l}(1-p^{-n-2l-3+2i}t^4)  \sum_{d \in {\mathcal U}(n-1,n-1-2l,d_0)} P_{2l}^{(0)}(d_0d,\omega,\chi(d),X,tY^{-1/2}) \right.$$
$$ \times { (p^{2l-1}Y^2)^{(n-2l-2)/2} \prod_{i=1}^{l-1} (1-p^{2i-1}Y^2)  p^{l-1/2}\chi(d)Y(1+\chi(d)Yp^{l-1/2}) \over 2\phi_{(n-2l-2)/2}(p^{-2})}$$
$$+\sum_{l=0}^{(n-2)/2}  \prod_{i=1}^{l}(1-p^{-n-2l-3+2i}t^4) P_{2l+1}^{(1)}(d_0,\omega,1,X,tY^{-1/2}) $$
$$ \times  \left. { (p^{2l-1}Y^2)^{(n-2l-2)/2} \prod_{i=1}^{l} (1-p^{2i-1}Y^2) (1+p^{-2l-2}t^2) \over \phi_{(n-2l-2)/2}(p^{-2})} \right\}.$$

\noindent
{\rm (3)} Suppose that $\nu(d_0) >0$ and $\omega=\varepsilon.$ Then   
$$\widetilde R_{n-1}(d_0,\omega,X,Y,t)=(1-\xi_0 p^{-1/2}Y)\prod_{i=1}^{n/2}(1-p^{-2n+2i-2}t^4)$$
$$ \times \sum_{l=0}^{(n-2)/2} P_{2l+1}^{(1)}(d_0,\omega,1,X,tY^{-1/2})   $$
$$ \times {(p^{2l+1}Y^2)^{(n-2l-2)/2} \prod_{i=1}^{l} (1-p^{2i-1}Y^2)\prod_{i=1}^{l}(1-p^{-2l-n+2i-3}t^4) \over \phi_{(n-2-2l)/2}(p^{-2})}.$$

\end{xcor}

\subsection{Explicit formulas of formal power series of Rankin-Selberg type}
 
 \noindent
  { }
  
  \bigskip
  We prove our main result in this section.
  
\begin{thms}    
 Let $d_0 \in {\mathcal F}_p$ and put $\xi_0=\chi(d_0).$

\noindent
  {\rm (1)}  We have
 $$H_{n-1}(d_0,\iota,X,Y,t)$$
 $$=\phi_{(n-2)/2}(p^{-2})^{-1}(p^{-1}t)^{\nu(d_0)}(1-p^{-n}t^2)\prod_{i=1}^{n/2-1} (1-p^{-2n+2i}t^4) $$
 $$ \times { (1+p^{-2} t^2)(1+ p^{-3}\xi_0^2t^2) -p^{-5/2}t^2\xi_0(X+X^{-1}+Y+Y^{-1})\over    (1-p^{-2}XYt^2)(1-p^{-2}XY^{-1}t^2)(1-p^{-2}X^{-1}Yt^2)(1-p^{-2}X^{-1}Y^{-1}t^2)}$$
$$ \times  { 1 \over \prod_{i=1}^{n/2-1} (1-p^{-2i-1}XYt^2)(1-p^{-2i-1}XY^{-1}t^2)(1-p^{-2i-1}X^{-1}Yt^2)(1-p^{-2i-1}X^{-1}Y^{-1}t^2)}.$$ 
 
 \noindent
 {\rm (2)} We have 
 $$H_{n-1}(d_0,\varepsilon,X,Y,t)$$
$$=\phi_{(n-2)/2}(p^{-2})^{-1}(1-p^{-n}t^2)\prod_{i=1}^{n/2-1} (1-p^{-2n+2i}t^4)(tp^{-n/2})^{\nu(d_0)}$$
 $$ \times{ (1+p^{-n} t^2)(1+ p^{-n-1}\xi_0^2t^2) -p^{-1/2-n}t^2\xi_0(X+X^{-1}+Y+Y^{-1})\over    (1-p^{-n}XYt^2)(1-p^{-n}XY^{-1}t^2)(1-p^{-n}X^{-1}Yt^2)(1-p^{-n}X^{-1}Y^{-1}t^2)}$$
 $$ \times {1 \over \prod_{i=1}^{n/2-1} (1-p^{-2i}XYt^2)(1-p^{-2i}XY^{-1}t^2)(1-p^{-2i}X^{-1}Yt^2)(1-p^{-2i}X^{-1}Y^{-1}t^2)}.$$
\end{thms}
 
\begin{proof} First suppose  that  $\omega=\iota.$ For an integer $l$ put
$$V(l,X,Y,t)$$
$$=(1-t^2p^{-2}XY^{-1})(1-t^2p^{-2}X^{-1}Y^{-1}) \prod_{i=1}^{l}(1-t^2p^{-2i-1}XY^{-1})(1-t^2p^{-2i-1}X^{-1}Y^{-1}).$$
For $d \in {\mathcal U},$ put $\eta_d=\chi(d).$ Then  by  Theorem 5.2.2, and (1) of Corollary to Theorem 5.2.8, we have
$$\widetilde R_{n-1}(d_0,\omega,X,Y,t)=(1-p^{-n}t^2) \prod_{i=1}^{(n-2)/2}(1-p^{-2n+2i}t^4) \left\{ { (p^{-1}Y^2)^{(n-1)/2} \xi_0  \over \phi_{(n-2)/2}(p^{-2})} \right.$$
$$+{(p^{-1}Y^2)^{(n-2)/2}(1-p^{-1/2}\xi_0 Y)   (1+p^{-2}t^2)(p^{-1}tY^{-1/2})^{\nu(d_0)} (1-p^{-5/2} \xi_0 t^2 Y^{-1}) \over (1-t^2p^{-2}XY^{-1}) (1-t^2p^{-2}X^{-1}Y^{-1})}$$
$$ + \sum_{l=1}^{(n-2)/2}\sum_{d \in {\mathcal U}(n-1,n-1-2l,d_0)} {\prod_{i=1}^{l}(1-p^{-n-2l-3+2i}t^4) (p^{-1}tY^{-1/2})^{\nu(d_0)}S_{2l}^{(0)}(d_0d,\iota,\eta_d,X,Y,t^2) \over V(l,X,Y,t) }$$
$$\left. +\sum_{l=1}^{(n-2)/2}  {\prod_{i=2}^{l}(1-p^{-n-2l-3+2i}t^4) (p^{-1}tY^{-1/2})^{\nu(d_0)}S_{2l+1}^{(1)}(d_0,\iota,X,Y,t^2) \over V(l-1,X,Y,t) } \right\},$$
where $S_{2r}^{(0)}(d_0 d, \iota, \eta_d,X,Y,t)$ and $S_{2r+1}^{(1)}(d_0,\iota,X,Y,t)$ are  polynomials in $t$ of degree at most $2$. We note that $\xi_0=0$ if $\nu(d_0)>0.$ Hence  $\widetilde R_{n-1}(d_0,\iota,X,Y,t)$ can be expressed as
$$\widetilde R_{n-1}(d_0,\iota,X,Y,t)$$
$$={(1-p^{-n}t^2) \prod_{i=1}^{(n-2)/2}(1-p^{-n-2i}t^4) (p^{-1}tY^{-1/2})^{\nu(d_0)}S(d_0,\iota, X,Y,t^2) \over \phi_{(n-2)/2}(p^{-2}) V((n-2)/2,X,Y,t)},$$
where $S(d_0, \iota,X,Y,t)$ is a polynomial in  $t$ of degree at most $n.$ Moreover it can be expressed as  
\\
${\rm (D)} \qquad S(d_0, \iota,X,Y,t^2) $
$$=\{{ (p^{-1}Y^2)^{(n-2)/2} \prod_{i=1}^{(n-2)/2} (1-t^2p^{-2i-1}XY^{-1}) (1-t^2p^{-2i-1}X^{-1}Y^{-1}) \over \phi_{(n-2)/2}(p^{-2})}$$
$$ \times \xi_0 p^{-1/2}Y  (1-t^2p^{-2}XY^{-1}) (1-t^2p^{-2}X^{-1}Y^{-1})$$
$$+{ (p^{-1}Y^2)^{(n-2)/2} \prod_{i=1}^{(n-2)/2} (1-t^2p^{-2i-1}XY^{-1}) (1-t^2p^{-2i-1}X^{-1}Y^{-1}) \over \phi_{(n-2)/2}(p^{-2})}$$
$$\times (1-p^{-1/2}\xi_0 Y)   (1+p^{-2}t^2) (1-p^{-5/2} \xi_0 t^2 Y^{-1}) \}$$
$$+(1-p^{-n-3}t^4)U(d_0,X,Y,\iota,t^2)$$
with $U(d_0,\iota,X,Y,t)$ a polynomial in  $t.$ 
Hence  by Theorem 5.1.4 we have 
$$H_{n-1}(d_0,\iota,X,Y,t)=(p^{-1}t)^{\nu(d_0)}(1-p^{-n}t^2)\prod_{i=1}^{n/2-1} (1-p^{-2n+2i}t^4) $$
$$ \times {S(d_0,\iota,X,Y,t^2) \over (1-p^{-2}XYt^2)(1-p^{-2}XY^{-1}t^2)(1-p^{-2}X^{-1}Yt^2)(1-p^{-2}X^{-1}Y^{-1}t^2) }$$
$$ \times {1 \over \prod_{i=1}^{n/2-1} (1-p^{-2i-1}XYt^2)(1-p^{-2i-1}XY^{-1}t^2)(1-p^{-2i-1}X^{-1}Yt^2)(1-p^{-2i-1}X^{-1}Y^{-1}t^2)} $$
$$ \times {1 \over \prod_{i=1}^{(n-2)/2}(1-p^{-2i}XYt^2)(1-p^{-2i}X^{-1}Yt^2)}.$$ 
Hence the power series $H_{n-1}(d_0,\iota,X,Y,t)$ is a rational function in $X,Y$ and $t$, and is invariant under the transformation $Y \mapsto Y^{-1}.$  This implies that the reduced denominator of the rational function  $H_{n-1}(d_0,\iota,X,Y,t)$  in $t$  is at most 
 $$(1-p^{-2}XYt^2)(1-p^{-2}XY^{-1}t^2)(1-p^{-2}X^{-1}Yt^2)(1-p^{-2}X^{-1}Y^{-1}t^2)$$
 $$ \times \prod_{i=1}^{n/2-1} (1-p^{-2i-1}XYt^2)(1-p^{-2i-1}XY^{-1}t^2)(1-p^{-2i-1}X^{-1}Yt^2)(1-p^{-2i-1}X^{-1}Y^{-1}t^2)$$
 and therefore we have
$${\rm (E)} \qquad   S(d_0,\iota,X,Y,t^2)=\sum_{i=0}^2 a_i(d_0,X,Y)t^{2i}{\prod_{i=1}^{(n-2)/2}(1-p^{-2i}XYt^2)(1-p^{-2i}X^{-1}Yt^2) \over \phi_{(n-2)/2}(p^{-2})}, $$
  where $a_i(d_0,X,Y) \ (i=0,1,2)$ is a  polynomial in $X+X^{-1}$ and $Y+Y^{-1}$.  First assume  $\nu(d_0)=0.$ Then we can easily see   $a_0(d_0,X,Y)=1.$ Then by substituting $\pm p^{(n+3)/4}$ for $t$ in  (D) and (E), and comparing them, we obtain
  $$1 \pm  a_1(d_0,X,Y)p^{(n+3)/2} +a_2(d_0,X,Y)p^{n+3} $$
  $$=1 \pm( p^{(n-3)/2}+p^{(n-1)/2} -\xi_0p^{n/2-1}(X+X^{-1}+Y+Y^{-1})) +p^{n-2}.$$
   Hence   
  $a_1(d_0,X,Y)=p^{-2}+p^{-3}-p^{-5/2}(X+X^{-1}+Y+Y^{-1})\xi_0$ and $a_2(d_0,X,Y)=p^{-5}.$ This proves the assertion in case $\nu(d_0)=0.$ Next assume $\nu(d_0) >0.$ Then in the same manner as above we have $a_0(d_0,X,Y)=1,$ and 
  $$1  \pm a_1(d_0,X,Y)p^{(n+3)/2}+a_2(d_0,X,Y)p^{n+3}= 1 \pm p^{(n-1)/2}$$
 Hence   $a_2(d_0,X,Y)=0$ and $a_1(d_0,X,Y)=p^{-2}.$  This proves the assertion in case $\nu(d_0)>0.$ 
  
Similarly the assertion for $\nu(d_0)=0$ and $\omega=\varepsilon$ can be proved. 
     Next suppose that $\nu(d_0) >0$ and $\omega=\varepsilon.$ 
Then the assertion can be proved similarly by using   Theorems 5.1.4 and 5.2.2, and (2) of Corollary to Theorem 5.2.8.
\end{proof}
 
\section{Proof of Theorem 3.2}
   Now we give an explicit form of $R(s,\sigma_{n-1}(\phi_{I_n(h),1}))$ for the first Fourier-Jacobi coefficient $\phi_{I_n(h),1}$ of the Duke-Imamo{\=g}lu-Ikeda lift. 

\begin{prop}
 Let $k$ and $n$ be positive even integers. Given a Hecke eigenform $ h \in {\textfrak S}_{k-n/2+1/2}^+(\varGamma_0(4)),$  let $f \in {\textfrak S}_{2k-n}(\varGamma^{(1)})$ be the primitive form as in Section 2. Then
\begin{eqnarray*}
\lefteqn{R(s, h)=L(2s-2k+n+1,f,{\rm Ad})
\sum_{d_0 \in {\mathcal F}^{(-1)^{n/2}}} |c_h(|d_0|)|^2 |d_0|^{-s} } \\
&& \hspace*{-5mm}\times \prod_{p}\{(1+p^{-2s+2k-n-1})(1+p^{-2s+2k-n-2}\chi_{p}(d_0)^2)-2p^{-2s+2k-n-3/2}\chi_{p}(d_0)c_f(p)\}. \end{eqnarray*}
\end{prop}

\begin{proof} For any prime number $p$ we have
$$\sum_{r=0}^{\infty} c_h(|d_0|p^{2r})p^{-2rs}=c_h(|d_0|){ 1- p^{-2s+k-n/2-1}\left({d_0 \over p}\right) \over   (1-p^{k-n/2+1-2s}\alpha_p) (1-p^{k-n/2+1-2s}\alpha_p^{-1})}.$$ 
Then by using the same method as in the proof of [\cite{Sh1}, Lemma1] we can show that
$$\sum_{r=0}^{\infty} c_h(|d_0|p^{2r})^2p^{-2rs}=c_h(|d_0|)^2$$
$$ \times {(1+p^{-2s+2k-n-1})(1+p^{-2s+2k-n-2}\chi_{p}(d_0)^2)-2p^{-2s+2k-n-3/2}\chi_{p}(d_0)a(p) \over
 (1-p^{-2s+2k-n-1}\alpha_p^2)(1-p^{-2s+2k-n-1}\alpha_p^{-2})(1-p^{-2s+2k-n-1})}$$
This proves the assertion.
\end{proof}
\noindent
{\bf Remark.} If $n \equiv 2 \ {\rm mod} \ 4,$ this can also be proved by Theorems 4.2 and 5.3.1.

\begin{thm}
Let $k$ and $n$ be positive even integers. Given a Hecke eigenform $h \in {\textfrak S}_{k-n/2+1/2}^+(\varGamma_0(4)),$ let $ f \in {\textfrak S}_{2k-n}(\varGamma^{(1)})$ and $\phi_{I_n(h),1} \in {J_{k,\, 1}^{\, {\rm cusp}}}(\varGamma^{(n-1),J}) $ be as in Section 2 and Section 3, respectively. Put $\lambda_n={e_{n-1} \over 2}\prod_{i=1}^{n/2-1} \widetilde {\xi}(2i).$ 
Then, we  have
$$R(s,\sigma_{n-1}(\phi_{I_n(h),1}))=\lambda_n2^{(-s-1/2)(n-2)}\zeta(2s+n-2k+1)^{-1}\prod_{i=1}^{{n-2 \over 2}}\zeta(4s+2n-4k+2-2i)^{-1}$$
$$\times \{R(s-n/2+1, h) \zeta(2s-2k+3)\prod_{i=1}^{{n-2 \over 2}}L(2s-2k+2i+2,f,{\rm Ad})\zeta(2s-2k+2i+2)$$
$$+ (-1)^{n(n-2)/8}R(s, h) \zeta(2s-2k+n+1)\prod_{i=1}^{{n-2 \over 2}}L(2s-2k+2i+1,f,{\rm Ad})\zeta(2s-2k+2i+1) \}.$$
\end{thm}

\begin{proof} The assertion follows directly from Theorems 4.2 and 5.3.1, and Proposition 6.1.
\end{proof}

\begin{proof} [{\bf Proof of Theorem 3.2.}]  The assertion trivially holds if $n=2.$ Suppose that $n \ge 4.$ By Theorem 6.2 we have\\
${\rm (F)} \qquad \displaystyle {\mathcal R}(s,\sigma_{n-1}(\phi_{I_n(h),1}))=\prod_{i=1}^{n/2-1}\widetilde \xi(2i) 2^{(-s-1/2)(n-2)}{\mathcal T}(s)$
$$\times \left\{{\mathcal U}(s)^{-1}{\mathcal R}(s-n/2+1, h)\prod_{i=1}^{{n-2 \over 2}}\widetilde \Lambda(2s-2k+2i+2,f,{\rm Ad}) \xi(2s-2k+2i+2)\right.$$
$$\left.+(-1)^{n(n-2)/8}R(s, h)\prod_{i=1}^{{n-2 \over 2}}L(2s-2k+2i+1,f,{\rm Ad})\zeta(2s-2k+2i+1)\right\},$$
where 
$${\mathcal T}(s)=\Gamma_{\bf R}(2s+n-2k+1)\prod_{i=1}^{(n-2)/2}\Gamma_{\bf R}(4s+2n-4k+2-2i) \prod_{i=1}^{n-1}\Gamma_{\bf R}(2s-i+1),$$
and
$${\mathcal U}(s)=\Gamma_{\bf R}(2s-2k+3)\Gamma_{\bf R}(2s-n+2)$$
$$ \times \prod_{i=1}^{(n-2)/2}(\Gamma_{\bf C}(2s-2k+2i+2)\Gamma_{\bf C}(2s-n+2i+1)\Gamma_{\bf R}(2s-2k+2i+2)).$$
We note that ${\mathcal R}(s, h)$  is holomorphic at $s=k-1/2.$ Thus by taking the residue of the both-sides of (F) at $s=k-1/2$ , we get  
$${\rm Res}_{s=k-1/2}{\mathcal R}(s,\sigma_{n-1}(\phi_{I_n(h),1}))=2^{-k(n-2)}\prod_{i=1}^{n/2-1}\widetilde \xi(2i) {{\mathcal T}(k-1/2) \over {\mathcal U}(k-1/2)}$$
$$ \times {\rm Res}_{s=k-n/2+1/2} {\mathcal R}(s, h) \prod_{i=1}^{{n-2 \over 2}} \widetilde \Lambda(2i+1,f,{\rm Ad})\xi(2i+1).$$
We easily see that 
$${{\mathcal T}(k-1/2) \over {\mathcal U}(k-1/2)}=2^{(n-1)(n-2)/2}.$$ 
By Theorem 1 in \cite{K-Z}, 
we have
$${\rm Res}_{s=k-n/2+1/2}{\mathcal R}(s, h)=2^{2k-n}\langle  h, h \rangle .$$
Thus we complete the proof.
  \end{proof}

\bigskip

\noindent
\author{Hidenori KATSURADA \\[1mm]
Muroran Institute of Technology \\
 27-1 Mizumoto, Muroran, 050-8585, Japan \\
E-mail: \verb+hidenori@mmm.muroran-it.ac.jp+ \\[2.5mm]
 and \\[2.5mm]
\noindent
 Hisa-aki KAWAMURA \\[1mm] 
 Department of Mathematics, Graduate School of Science, Hiroshima University \\
 1-3-1 Kagamiyama, Higashi-Hiroshima, 739-8526 JAPAN \\
E-mail: \verb+hisa@hiroshima-u.ac.jp+}


\begin{thebibliography}{9999999}
\addcontentsline{toc}{section}{Bibliography}
\bibitem 
{A}A. N. Andrianov, 
\textit{Quadratic forms and Hecke operators}, 
Grundl. Math. Wiss., 286, Springer-Verlag, Berlin, 1987. 
\bibitem 
{Bo}S. B{\"o}cherer, 
\textit{Eine Rationalit\"atsatz f\"ur formale Heckereihen zur Siegelschen Modulgruppe}, 
Abh. Math. Sem. Univ. Hamburg {\bf 56} (1986), 35--47. 
\bibitem 
{B-D-S} S. B{\"o}cherer, N. Dummigan, and R. Schulze-Pillot, 
\textit{ Yoshida lifts and Selmer groups},
J. Math. Soc. Japan {\bf 64}(2012), 1353-1405.
\bibitem 
{B-S}S. B{\"o}cherer and F. Sato, 
\textit{Rationality of certain formal power series related to local densities}, 
Comment. Math. Univ. St.Paul. {\bf 36} (1987), 53--86. 
\bibitem 
{Br}J. Brown, 
\textit{Saito-Kurokawa lifts and applications to the Bloch-Kato conjecture}, 
Compos. Math. {\bf 143} (2007), no. 2, 290--322. 
\bibitem 
{C-K}Y. Choie and W. Kohnen, 
\textit{On the Petersson norm of certain Siegel modular forms}, 
Ramanujan J. {\bf 7} (2003), 45--48. 
\bibitem 
{D-H-I}K. Doi, H. Hida and H. Ishii, 
\textit{Discriminant of Hecke fields and twisted adjoint L-values for $GL(2)$}, 
Invent. Math. {\bf 134} (1998), 547--577.
\bibitem 
{F}M. Furusawa, 
\textit{On Petersson norms for some liftings}, 
Math. Ann. {\bf 248} (1984), 543--548. 
\bibitem 
{Ib}T. Ibukiyama, 
\textit{On Jacobi forms and Siegel modular forms of half integral weights}, 
Comment. Math. Univ. St.Paul. {\bf 41} (1992), no. 2, 109--124. 
\bibitem 
{I-K1}T. Ibukiyama and H. Katsurada, 
\textit{An explicit formula for Koecher-Maa{\ss} Dirichlet series for Eisenstein series of Klingen type}, 
J. Number Theory, {\bf 102} (2003), 223--256.
\bibitem 
{I-K2}T. Ibukiyama and H. Katsurada, 
\textit{An explicit formula for Koecher-Maa{\ss} Dirichlet series for the Ikeda lifting}, 
Abh. Math. Sem. Hamburg {\bf 74} (2004), 101--121.
\bibitem 
{I-K3}T. Ibukiyama and H. Katsurada, 
\textit{Koecher-Maa{\ss} series for real analytic Siegel Eisenstein series}, 
Automorphic forms and zeta functions, 170--197, World Sci. Publ., Hackensack, NJ, 2006. 
\bibitem 
{I-S}T. Ibukiyama and H. Saito, 
\textit{On zeta functions associated to symmetric matrices. I. An explicit form of zeta functions}, 
Amer. J. Math. {\bf 117} (1995), 1097--1155. 
\bibitem 
{Ik1}T. Ikeda, 
\textit{On the lifting of elliptic modular forms to 
Siegel cusp forms of degree $2n$}, 
Ann. of Math. {\bf 154} (2001), no. 3, 641--681. 
\bibitem 
{Ik2}T. Ikeda, 
\textit{Pullback of the lifting of elliptic cusp forms and Miyawaki's conjecture}, 
Duke Math. J. {\bf 131} (2006), no. 3, 469--497.
\bibitem 
{Ik3}T. Ikeda, \textit{On the lifting of hermitian modular forms}
Compositio Math. {\bf 144} (2008) 1107-1154.
\bibitem 
{Kal}V. L. Kalinin, 
\textit{Analytic properties of the convolution products of genus $g$}, 
Math. USSR Sbornik {\bf 48} (1984), 193--200. 
\bibitem 
{Kat1}H. Katsurada, 
\textit{An explicit formula for Siegel series}, 
Amer. J. Math. {\bf 121} (1999), 415--452. 
\bibitem 
{Kat2}H. Katsurada, 
\textit{Congruence of Siegel modular forms  and  special values of their standard zeta functions}, 
Math. Z. {\bf 259} (2008), 97--111. 
\bibitem 
{Kat4} H. Katsurada,
\textit{Exact standard zeta values of Siegel modular forms}
Experiment. Math. {\bf 19}(2010),65-76.
\bibitem 
{Kat3}H. Katsurada, 
\textit{Congruence between Duke-Imamo{\=g}lu-Ikeda lifts and non-Duke-Imamo{\=g}lu-Ikeda lifts}, 
preprint, arXiv:1101.3377v1[math. NT].
\bibitem 
{Kat5} H. Katsurada, 
\textit{Explicit formulas for the twisted Koecher-Maass series of the Duke-Imamoglu-Ikeda lift and their applications}, to appear in Math. Z, arXiv:1310.1544[math.NT].
\bibitem 
{K-K1}H. Katsurada and H. Kawamura, 
\textit{A certain Dirichlet series of Rankin-Selberg type associated with the Ikeda lifting}, 
J. Number Theory {\bf 128} (2008), 2025--2052. 
\bibitem 
{K-K3}H. Katsurada and H. Kawamura, 
\textit{On Andrianov type identity for a power series attached to Jacobi forms and its applications}, 
Acta Arith. {\bf 145}(2010), 233--265.
\bibitem 
{K-K4}H. Katsurada and H. Kawamura, 
\textit{Koecher-Maa{\ss}   series of a certain half-integral weight modular form related to the Duke-Imamo{\=g}lu-Ikeda lift}, To appear in Acta. Arith, arXiv:1309.2065[math. NT].
\bibitem 
{Ki1}Y. Kitaoka, 
\textit{Dirichlet series in the theory of Siegel modular forms}, 
Nagoya Math. J. {\bf 95} (1984), 73--84. 
\bibitem 
{Ki2}Y. Kitaoka, 
\textit{Arithmetic of quadratic forms}, 
Cambridge Tracts in Mathematics, 106. Cambridge University Press, Cambridge, 1993
\bibitem 
{Ko}W. Kohnen, 
\textit{Modular forms of half-integral weight on $\Gamma_0(4)$}, 
Math. Ann. {\bf 248} (1980), 249--266. 
\bibitem 
{Kr}A. Krieg, 
\textit{A Dirichlet series for modular forms of degree $n$}, 
Acta Arith. {\bf 59} (1991), 243--259. 
\bibitem 
{K-S}W. Kohnen and N.-P. Skoruppa, 
\textit{A certain Dirichlet series attached to Siegel modular forms of degree 2}, 
Invent. Math. {\bf 95} (1989), 541--558.
\bibitem 
{K-Z}W. Kohnen and D. Zagier, 
\textit{Values of $L$-series of modular forms at the center of the critical strip}, 
Invent. Math. {\bf 64} (1981), 175--198. 
\bibitem 
{M-S}A. Murase and T. Sugano, 
\textit{Inner product formula for Kudla lift}, 
Automorphic forms and zeta functions, 280--313, World Sci. Publ., Hackensack, NJ, 2006. 
\bibitem 
{O}T. Oda, 
\textit{On the poles of Andrianov $L$-functions}, 
Math. Ann. {\bf 256} (1981), 323--340.
\bibitem 
{Ra}S. Rallis, 
\textit{$L$-functions and Oscillator representation}, 
Lecture Notes in Math., Vol. 1245, Springer-Verlag, Berlin, 1987. 
\bibitem 
{Schul}R. Schulze-Pillot, \textit{Local theta correspondence and the theta lifting of Duke-Imamoglu and Ikeda}, Osaka J. Math. {\bf 45}(2008), 965-971.
\bibitem 
{Sh1}G. Shimura, 
\textit{The special values of the zeta functions associated with cusp forms}, 
Comm. Pure Appl. Math. {\bf 29} (1976), no. 6, 783--804
\bibitem 
{Sh2}G. Shimura
\textit{Arithmeticity in the theory of automorphic forms}, 
Mathematical Surveys and Monographs, 82, Amer. Math. Soc., 2000.
\bibitem 
{Y}T. Yamazaki, 
\textit{Rankin-Selberg method for Siegel cusp forms}, 
Nagoya Math. J. {\bf 120} (1990), 35--49. 
\bibitem 
{Za}D. Zagier, 
\textit{Modular forms whose Fourier coefficients involve zeta functions of quadratic fields}, Modular functions of one variable, VI (Proc. Second Internat. Conf., Univ. Bonn, Bonn, 1976), pp. 105--169. 
Lecture Notes in Math., Vol. 627, Springer, Berlin, 1977. 
\bibitem 
{Zh}V. G. Zhuravl{\"e}v, 
\textit{Euler expansions of theta-transformations of Siegel modular forms of half integer weight and their analytic properties}, 
Math. USSR Sbornik {\bf 51} (1985), 169--190. 
\end{thebibliography}
\end{document}